\documentclass[preprint]{article}
\RequirePackage[OT1]{fontenc}
\usepackage{amsmath,amssymb}
\usepackage{epsfig}
\usepackage{graphicx}
\usepackage{epstopdf,color}
\usepackage{multirow}
\usepackage{listings}
\usepackage{txfonts}
\usepackage{indentfirst}
\RequirePackage{natbib}
\RequirePackage[colorlinks,citecolor=blue,urlcolor=blue]{hyperref}
\allowdisplaybreaks[4]
\usepackage{euscript}
\numberwithin{equation}{section} 
\usepackage{xr}

\usepackage[a4paper]{geometry}
\newtheorem{remark}{Remark}[section] 
\newtheorem{theorem}{Theorem}[section] 
\newtheorem{lemma}{Lemma}[section] 
\numberwithin{equation}{section}

 \newcommand{\G}{{\mathbf G}}

\newcommand{\Y}{{\mathbf Y}} 
\newcommand{\Z}{{\mathbf Z}} 
\newcommand{\T}{{\mathbf T}} 
\newcommand{\A}{{\mathbf A}} 
\newcommand{\B}{{\mathbf B}} 
\newcommand{\C}{{\mathbf C}} 
\newcommand{\D}{{\mathbf D}} 
\newcommand{\bS}{{\mathbf S}}
\newcommand{\W}{{\mathbf W}}
\newcommand{\I}{{\mathbf I}}
\newcommand{\s}{{\mathbf s}} 

\newcommand{\bx}{{\mathbf x}} 
\newcommand{\y}{{\mathbf y}} 
\newcommand{\z}{{\mathbf z}} 
\newcommand{\m}{{\mathbf m}} 
 
\newcommand{\br}{{\mathbf r}} 
\newcommand{\zt}{{\tilde z}} 
\newcommand{\bv}{{\mathbf v}}

\newcommand{\CC}{\mathbb{C}}

\newcommand{\um}{\underline{m}} 
 
\newcommand{\uS}{\underline{S}} 
\newcommand{\uT}{\underline{T}} 
\newcommand{\de}{\delta} 
\newcommand{\E}{{\mathbb{E}}}
  
\newcommand{\Cov}{{\rm Cov}} 

\newcommand{\U}{{\mathbf U}} 
\newcommand{\bV}{{\mathbf V}} 
\newcommand{\bM}{{\mathbf M}}

\newcommand{\R}{{\tilde {\mathbf R}}} 
\newcommand{\tr}{\mathop{\text{\rm tr}}}

\newcommand{\bxi}{{\boldsymbol \xi}}

\newcommand{\bSigma}{{\boldsymbol \Sigma}}

\definecolor{darkblue}{rgb}{0,0.08,0.45}

\begin{document}

\begin{center}
{\large \bf On eigenvalues of a high-dimensional spatial-sign covariance matrix}
\end{center}

\vskip 0.5cm
\begin{center}
	Weiming Li$^a$, Qinwen Wang$^b$, Jianfeng Yao$^c$ and Wang Zhou$^d$\\
	\vskip 0.3cm
	$^a$Shanghai University of Finance and Economics, $^b$Fudan University, $^c$University of Hong Kong  and $^d$National University of Singapore
\end{center}

\begin{abstract} 
    This paper investigates limiting spectral
    properties of  a high-dimensional  sample spatial-sign
    covariance matrix 
    when both the dimension  and the sample size grow to
    infinity. The underlying population is general enough to
    cover the popular independent components model and the family of
    elliptical distributions.
    The first result of the paper shows that the empirical
    spectral distribution of a high dimensional
    sample spatial-sign  covariance matrix converges to a generalized 
    Mar\v{c}enko-Pastur distribution.
   Secondly, a new central limit theorem  for a class of  linear  spectral statistics of the covariance matrix is established under moment conditions.
\end{abstract}

\section{Introduction} 
Let $\bx_1,\ldots,\bx_n$ be a sequence of
{\em independent and identically distributed} (i.i.d.) observations  from a common population 
$\bx\in\mathbb R^p$  with known location vector $\m$
(mean or median).  The sample {\em  spatial-sign covariance matrix} (SSCM) is by definition
\begin{align*}
  \B_n=\frac{p}{n}\sum_{j=1}^n \frac{(\bx_j-\m)}{\|\bx_j-\m\|}\frac{(\bx_j-\m)'}{\|\bx_j-\m\|},
\end{align*}
where $||\cdot||$ denotes the Euclidean norm of a vector.
In \citet{L99} and \citet{Visuri00}, the authors demonstrated that the SSCM is able to  mitigate the impact of  extreme  outliers for the purpose  of robust principal components analysis.
 Since then, the SSCM   has been  widely adopted for robust statistical
inference where the sample data may exhibit heavy tails,   or  bear tail
dependence as in the case of elliptical distributions.  
Recent works concerning the properties of the SSCM and its applications include 
\citet{M14}, 
\citet{Durre14, Durre15a}, \citet{LiWangZou16}, \citet{Feng16}, \citet{feng2016} 
and \citet{Chakraborty17}.
Despite the popularity of the SSCM,
asymptotic behaviors of its eigenvalues are not fully developed when the dimension of the population $p$ diverges to infinity along with the sample size $n$, which greatly limits its application to high-dimensional data analysis.

This paper investigates the first and second order spectrum limits of sample SSCMs under the
Mar\v{c}enko-Pastur asymptotic regime \citep{MP67}, i.e.
$$ n \to \infty,\quad p=p(n)\to \infty,\quad p/n=c_n\to c \in(0, \infty),$$
which is commonly adopted in the literature of  random matrix theory. 
The underlying population $\bx$ considered here has a general structure,
  \begin{equation}\label{model}
    \bx =  \m + w \A^{\frac 12}\z, 
  \end{equation} 
where 
$\m\in\mathbb R^p$ is the location vector, 
$\A$ is a $p\times p$ deterministic and positive definite matrix, 
$w\in \mathbb R$ and $\z\in\mathbb R^{p}$ are two (possibly dependent) random quantities with certain moment conditions, see \eqref{eds} for detailed model illustration.
The generality of this model lies in that it 
 encompasses  the popular {\em independent components model}  
and the family of  {\em elliptical distributions}, which will be explained in detail in Section \ref{sec:model}.

The first result of the paper is a new generalized  {\em Mar\v{c}enko-Pastur} (MP)
law for the {\em empirical spectral distribution} (ESD) of $\B_n$.  
The MP law was originally introduced in \cite{MP67} for the limiting spectrum of {\em  sample covariance matrices} (SCMs), which was then refined and extended in several works, say \cite{Yin86}, \cite{S95} and \cite{bai08}.
With this knowledge, through a comparison between the matrix $\B_n$ and its associated SCM  $\bS_n\triangleq \sum_{j=1}^n\A^{1/2}\z_j\z_j'\A^{1/2}/n$, our result is derived by showing $||\B_n-\bS_n||$ converges to zero, almost surely, under finite $(4+\delta)$-th moment condition on the components of the vector $\z$. 
As a by-product, one may draw the same ``no eigenvalue" conclusion for $\B_n$ as that for $\bS_n$ already established in \cite{BS98}.

The second contribution of this paper
is  a new
CLT for general {\em linear spectral statistics}
(LSSs) of $\B_n$.
CLT for LSSs  of large random matrices has been actively studied in
recent decades in random matrix theory. 
Most of early works in this area concern Hermitian (symmetric) Wigner
matrices.
\cite{Johansson98} presented  a CLT for LSSs
of eigenvalues 
given their joint density for Gaussian-type random Hermitian
matrices.  Using the moment method,
\cite{Sinai98} derived a CLT for traces of analytic functions of Wigner-type
matrices and
\cite{AZ06}
obtained a CLT for a class of  band random matrices.
CLT for general Wigner matrix with arbitrary entries is first derived
in \cite{BY05}  via Stieltjes transforms establishing
the explicit formula for the mean and covariance functions of the
limiting Gaussian distribution of the LSSs.
A related approach using Gaussian interpolation for both Wigner
matrices and Wishart matrices is proposed in \cite{LytovaPastur09}. 
As for sample covariance matrices, 
{the  earliest work  dates back to \cite{Jonsson82}
  for Wishart matrices.  The seminal paper  \cite{BS04}
  established the CLT under the independent components model, which was
later extended in  \cite{PZ08} and \cite{Z15}.  Other extensions on
CLT for sample covariance matrices are
recently proposed in  \cite{H19} and \cite{H19a}
for  the  class of elliptical distributions.

From the technical point of view, for establishing our CLT, the structure of the sample SSCM under study is quite different from the 
commonly studied random matrix models in the existing literature.
Although the spatial-sign (or projective)  transform $(\bx-\m)/||\bx-\m||$  removes  the impact from  the scaling random variable
$w$, it   does introduce, at the same time, complex  non-linear correlations
among the $p$-coordinates of the transformed data through the
normalization by  $\|\bx-\m\|$. Such  new correlations  make the
analysis more intricate in high dimensions where $p\to\infty$.
Specifically, let us compare the situation with a sample
covariance matrix ${\mathbf S}_n=n^{-1}\sum_{j=1}^n (\bx_j-\m)(\bx_j-\m)'$
from the independent components  model (see \eqref{linear-trans}). Here the correlations
among the coordinates of a sample vector $\bx_j$ have  only one source, coming from the shape matrix $\A$.
In the case of SSCM, the correlations among
the coordinates of $(\bx_j-\m)/||\bx_j-\m||$ can originate from both the shape matrix $\A$ and the normalization by $\|\bx_j-\m\|$. Therefore,
a main task in our analysis is to find new approaches for
decoupling these two sources of correlation.
To this end, by giving an asymptotic expansion of  $1/||\bx_j-\m||$ to
certain order, we develop some new lemmas concerning the covariance and
stochastic order of certain quadratic forms, which turns out to be one of the cornerstones for establishing our new CLT (see Section \ref{less}).
Another technical innovation of the paper, compared to the classical
approach in \cite{BS04},  is that we introduce a new
and more straightforward method to find 
the limiting mean function of LSSs, see Step 3 in the proof given in
Section~\ref{sec:proofs}.

The rest of the paper is organized as follows. 
Section~\ref{sec:2}   presents our
main theoretical results  including both the convergence of the ESD of $\B_n$ and the
CLT for its linear spectral statistics. 
Proofs of these asymptotic conclusions are presented in Sections \ref{sec:proofs0} and \ref{sec:proofs}.
Some supporting lemmas and their proofs  are  relegated  into the Appendix.

\section{High-dimensional theory for eigenvalues of a sample SSCM} 
\label{sec:2} 

\subsection{Preliminary definitions}
\label{sec:concept}

Let ${\mathbf M}_p$ be a  $p\times p$ symmetric or Hermitian matrix with eigenvalues $(\lambda_j)_{1\le j\le p}$.
Its ESD  is by 
definition the probability measure
\[ 
  F^{{\mathbf M}_p}=\frac{1}{p}\sum_{j=1}^p\delta_{\lambda_j}, 
\] 
where  $\delta_b$ denotes  the Dirac mass at $b$. 
If the ESD sequence  $\{F^{{\mathbf M}_p}\}$ has a limit when $p\to\infty$, this limit is referred as the {\em limiting spectral distribution}
 (LSD). 
For a probability  measure $G$, its Stieltjes transform is defined as
$$m_G(z)=\int\frac{1}{x-z}dG(x),\quad z\in\CC^+,$$ 
where $\mathbb C^+\equiv\{z\in \mathbb C: \Im(z)>0\}$. 
This definition can be extended to the whole complex plane
except the support set of $G$. An inversion formula of the Stieltjes transform can be found in \cite{BSbook}. 

Two sequences of  $\mathbb{R}^k$-valued random vectors
  $(\bxi_n)$ and $({\boldsymbol\eta}_n)$   are called {\em asymptotically equal in
    distribution}, denoted as $\bxi_n  \stackrel{d}{\sim}  {\boldsymbol\eta}_n$,
  if for any Borel set $C\subset \mathbb{R}^k$,
  \[ \mathbb{P} (\bxi_n\in C) - \mathbb{P} ({\boldsymbol\eta}_n\in C) \to 0,
    \quad n\to \infty.
  \]

\subsection{Model assumptions}
\label{sec:model}
We consider a sequence of i.i.d.\  observations $\bx_1,\ldots,\bx_n$ generated from the model \eqref{model}, which 
 admit
  the following stochastic representation:
  \begin{equation}\label{eds} 
    \bx _j=  \m + w_j \A^{\frac 12}\z_j, \quad j=1,\ldots,n,
  \end{equation} 
where
  \begin{itemize}
  \item[(i)]  the location vector $\m\in\mathbb R^p$ is assumed to be known;
  \item[(ii)] the scalar random variable $w_j$ is real-valued having no mass at the origin, i.e. ${\rm P}(w_j\neq 0)=1$;
   \item[(iii)] the matrix $\A\in \mathbb R^{p\times p}$, referred as the {\em shape matrix} or {\em scatter matrix} of the population, is deterministic, positive definite, and normalized as $\tr(\A)=p$ for the identification in the  triple product $w_j\A^{1/2}\z_j$, since we can always move any scalar factor related to $\A$ into the scalar random variable $w_j$;
    \item[(iv)] the vector  $\z_j=(z_{1j},\ldots,z_{pj})' \in \mathbb R^{p}$
     is  an array of  i.i.d.\ standardized random variables, possibly dependent of $w_j$.
\end{itemize}    
     
Our main assumptions are as follows. 
\medskip

\noindent{\em Assumption}   (a). \quad 
Both the sample size $n$ and population dimension $p$ tend to infinity
in such a way that $n\to \infty, p=p(n)\to \infty$ and  $p/n=c_n\to c \in(0,\infty)$. 

\medskip 
\noindent{\em Assumption}   (b). \quad The ESD $H_p$ of the shape matrix $\A$ has bounded support, i.e. ${\rm Supp}(H_p)\subset [a,b]$ for some $a, b\in (0,\infty)$, and 
 converges weakly to a probability distribution $H$ as $p\to\infty$. 
 
 \medskip 
\noindent{\em Assumption}   (c). \quad The random variables $(z_{ij})$ are i.i.d.\ and satisfy 
  \begin{align*}
  \E(z_{ij})=0,\quad \E(z_{ij}^2)=1, \quad \E(z_{ij}^{4})=\tau,\quad \E|z_{ij}|^{4+\delta}<\infty,
 \end{align*}
 for some $\delta>0$.


 \medskip 

\begin{remark}
Recall that in the literature on high-dimensional SCMs, the following  independent components model  is
routinely  considered \citep{BS04,PZ08,Z15,YZBbook}
\begin{equation}\label{linear-trans} 
  \bx_j=\m+\sigma\A^{\frac 12} \z_j, 
\end{equation} 
where $\m, \A$, and $\z_j$ are the same as in model \eqref{eds} while $\sigma$ is a positive constant.
Clearly the model \eqref{linear-trans} is a particular case of  the 
model~\eqref{eds} where $\{w_j\}$ degenerate to the constant parameter $\sigma$.
	
\end{remark}

\begin{remark}
The model \eqref{eds} contains also the family
  of  elliptical distributions. 
  Indeed, a generalized  elliptically distributed sample $\bx_j$ has the form
  \begin{equation}
    \label{elliptical}
   \bx_j = \m+v_j \A^{\frac 12} \mathbf{u}_j,
  \end{equation}
  where $v_j$ is a scalar random variable, $\mathbf{u}_j$ is a random vector uniformly distributed on the unit sphere in
  $\mathbb{R}^p$.
Let $\mathbf{u}_j=\z_j/\|\z_j\|$ and $w_j=v_j/||\z_j||$, where $\z_j\sim
  \mathcal{N}({\bf 0},\I_p)$,   we thus have
    \[ \bx_j = \m+v_j \A^{\frac 12} \mathbf{u}_j = \m+w_j \A^{\frac 12} \z_j.
  \]
  Certainly the moment conditions in Assumption (c) are satisfied with $\tau=3$ for such standard Gaussian random vectors $\{\z_j\}$.
  Thus the generalized elliptical distributions described by \eqref{elliptical} are also special cases of our model \eqref{eds}.
	
\end{remark}


\subsection{Sample SSCM and its limiting spectral distribution} 
\label{sec:main}

Let $\s(\y)= I_{(\y\neq 0)}\y/||\y||$ be  the
spatial-sign function projecting the vector $\y$  onto the unit
sphere. 
Then the sample SSCM $\B_n$ formed by the sample $\{\bx_j\}$ can be written as 
\begin{equation}
  \label{eq:An1}
  \B_n
  =\frac{p}{n}\sum_{j=1}^n \s(\bx_j-\m)\s(\bx_j-\m)'.
\end{equation}
Our first result is concerned with the convergence of the ESD
$F^{\B_n}$ of the sample SSCM $\B_n$.

\begin{theorem}\label{lsd} 
	Suppose that Assumptions (a)-(c) hold. Then, almost surely, the empirical spectral distribution $F^{\B_n}$  converges weakly to a probability distribution $F^{c, H}$, 
	whose Stieltjes transform $m=m(z)$ is the unique  solution to the equation 
	\begin{eqnarray}\label{mp1} 
	m=\int\frac{1}{t(1-c-czm)-z}d H(t)~,\quad z\in\CC^+, 
	\end{eqnarray} 
	in the set $\{m\in \mathbb C:  -(1-c)/z+cm\in{\mathbb C^+}\}$.
\end{theorem}

Theorem \ref{lsd} demonstrates that the ESD $F^{\B_n}$ converges to the generalized MP law $F^{c,H}$ \citep{MP67} defined through the equation \eqref{mp1}. 
Let $\underline{F}^{c,
  H}= cF^{c,H} + (1-c)\de_0$ be the companion distribution of $F^{c,H}$ and $\um=\um(z)$ be the Stieltjes transform of $\underline{F}^{c,
  H}$.  Then \eqref{mp1} can be rewritten as
\begin{equation}  \label{mp} 
z  =  - \frac1 {\um}  +  c \int\!\frac{t}{1+t\um} d H(t)~,\quad z\in\CC^+,
\end{equation} 
see \cite{S95}.
For procedures on numerically finding the density function of $F^{c, H}$ and its support set  from \eqref{mp1} or \eqref{mp}, one is referred to  \cite{BSbook}. 
 The proof of this theorem is presented in Section \ref{sec:proofs0}.

\subsection{CLT for linear spectral statistics of $\B_{n}$}

In this section, we   study  the fluctuation of LSSs of $\B_n$. Given a measurable function $f$, the LSS associated with $f$ and $\B_n$ is 
\begin{eqnarray}\label{lss} 
 \int f(x) dF^{\B_{n}}(x).
\end{eqnarray} 
 To centralize this statistic, we need to introduce a matrix $\T$ that is closely related to the shape matrix $\A$, i.e.
\begin{align}\label{tre}
\T=
\A-\frac{2}{p}\A^2
-\frac{\tau-3}{p}\A^{\frac 12}{\rm diag} (\A)\A^{\frac 12}
+\left(\frac{2}{p^2}\tr \A^2+\frac{\tau-3}{p^2}\tr(\A\circ\A)\right)\A,
\end{align}
where $``\circ"$ denotes the Hadamard product of two matrices.
This matrix is actually an approximation of the population SSCM $\bSigma\triangleq p\E(\bx-\m)(\bx-\m)'/||\bx-\m||^2$. Under certain conditions, we have the spectral norm $||\bSigma-\T||=o(p^{-1})$, see Lemma \ref{st}.
Note that for an elliptical distribution, see the model \eqref{elliptical}, the population SSCM $\bSigma$ and the shape matrix $\A$ share the same eigenvectors and their eigenvalues have
a one-to-one correspondence, which can be represented through certain integrals, see \citet{durre2016eigenvalues}.
Our approximation is however explicit and is not restricted to elliptical distributions.

Let $\tilde{H}_p$ be the ESD of the matrix $\T$ defined in \eqref{tre}, $\um_0(z)$ be the  finite-horizon proxy for the limiting Stieltjes transform $\underline{m}(z)$ in \eqref{mp}, i.e. the solution to
\begin{equation}  \label{mp0} 
z  =  - \frac1 {\um_0(z)}  +  c_n \int\!\frac{t}{1+t\um_0(z)} d \tilde H_p(t)~,\quad z\in\CC^+.
\end{equation} 
This Stieltjes transform $\um_0(z)$ uniquely defines a distribution, denoted by  $F^{c_n, \tilde{H}_p}$, through 
\begin{align}\label{um0}
\um_0(z)=-\frac{1-c_n}{z}+c_n\int \frac{1}{x-z}dF^{c_n, \tilde{H}_p}(x).
\end{align}
By means of this distribution,
The LSS in \eqref{lss} can be centralized as
\begin{eqnarray*} 
  G_n(f)\triangleq \int f(x) dG_{n}(x)=\int f(x) d[F^{\B_{n}}(x)-F^{c_n,  \tilde{H}_p}(x)]. 
\end{eqnarray*} 

We note that from \eqref{tre}, for the first order asymptotic of $G_n(f)$, one may replace the ESD $\tilde H_p$ of $\T$ by the ESD $H_p$ of $\A$ in the definition \eqref{mp0} of $\um_0(z)$ since the two matrices share the same LSD. However, for the second order asymptotic, the difference of $\T$ and  $\A$ will contribute when the shape matrix $\A$ is not identity.
In addition, three other auxiliary quantities defined as below will also contribute to the fluctuation of $G_n(f)$,
\begin{align}
  &\zeta_p=\frac{1}{p}\tr[\T\circ\T],\nonumber\\
  &h_p(z)=\frac{1}{p}\tr[\T^{\frac 12}(\T-zI)^{-1}\T^{\frac 12}\circ\T],\nonumber\\
  &g_p(z,\tilde z)=\frac{1}{p}\tr\left[\Big(\T^{\frac 12}(\T-zI)^{-1}\T^{\frac 12}\Big)\circ\Big(\T^{\frac 12}(\T-\tilde zI)^{-1}\T^{\frac 12}\Big)\right],
   \label{hg}
\end{align}
where $z$ and $\tilde z$ are two complex numbers in ${\mathbb C}^+$. 
These quantities  depend not only on the eigenvalues of $\A$, but also
on  its  eigenvectors.

\begin{theorem}\label{clt} 
  Suppose that Assumptions (a)-(c) hold with $\delta=1$.  Let $f_1,\ldots, f_k$ be $k$
  functions analytic on an open set that includes the interval 
  \begin{equation*}
  I_c=\left[\liminf_{p\rightarrow\infty}\lambda_{\min}^{\T}\delta_{(0,1)}(c)(1-\sqrt{c})^2,\quad \limsup_{p\rightarrow\infty}\lambda_{\max}^{\T}(1+\sqrt{c})^2\right].
  \end{equation*}
  Also let 
  $$
  \Y_n=p\left\{ G_n(f_1),\ldots, G_n(f_k)\right\}
  $$
  be the vector of $k$ normalized LSSs with respect to $f_1,\ldots,f_k$.
    Then $\Y_n$ is asymptotically equal in distribution to a 
    $k$-dimensional Gaussian random vector
    $\bxi_n=(\xi_{n1},\ldots,\xi_{nk})$
    with mean function 
  \begin{align*} 
    {\E}(\xi_{nj}) 
    &=-\frac{1}{2\pi\rm i}\oint_{\mathcal C_1} f_j(z)\left[\mu_1(z)+(\tau-3)\mu_2(z)\right]dz,
  \end{align*}
  where 
  \begin{align*} 
    \mu_1(z)
    =&\int\frac{c_n(\um_0'(z)t)^2d \tilde H_p(t)}{\um_0(z)(1+\um_0(z) t)^3}
       -\int\frac{2\um_0'(z)(1+z\um_0(z))t^2d \tilde H_p(t)}{(1+\um_0(z) t)^2}\\
              &+\int\frac{ (\tr(\A^2/p)t-t^2)d \tilde H_p(t)}{1+\um_0(z) t}\int\frac{2c_n\um_0(z)\um_0'(z)td \tilde H_p(t)}{(1+\um_0(z) t)^2},\\
    \mu_2(z)=
              &\frac{c_n\um_0'(z)}{\um_0^2(z)}g'_{p,z}\left(\frac{-1}{\um_0(z)},\frac{-1}{\um_0(z)}\right)
                +\zeta_p\int\frac{(1+z\um_0(z))t\um_0'(z) d \tilde H_p(t)}{(1+\um_0(z) t)^2}\\
              &-\frac{(1+z\um_0(z))\um_0'(z)}{\um_0^2(z)}h_p'\left(\frac{-1}{\um_0(z)}\right)
                -\int\frac{c_n\um_0'(z)t d \tilde H_p(t)}{(1+\um_0(z) t)^2}h_p\left(\frac{-1}{\um_0(z)}\right),
  \end{align*}
  and  covariance function 
  \begin{align*} 
    {\rm Cov}\left( \xi_{nj}, \xi_{n\ell}\right) 
    =&-\frac{1}{4\pi^2}\oint_{\mathcal C_1}\oint_{\mathcal C_2}f_j(z)f_{\ell}(\zt)\left[\sigma_1(z,\zt)+(\tau-3)\sigma_2(z,\zt)\right]dzd\zt,
  \end{align*} 
  where
  \begin{align*}
    \sigma_1(z,\zt)=&\frac{2\partial^2}{\partial z\partial \zt}\bigg[\log\frac{\um_0(z)-\um_0(\zt)}{\um_0(z)\um_0(\zt)(z-\zt)}+\left(\frac{\tr(\T^2)}{pc_n}+\frac{1}{c_n\um_0(z)}+\frac{1}{c_n\um_0(\zt)}\right)\\
                    &\times(1+z\um_0(z))(1+\zt\um_0(\zt))-z\um_0(z)-\zt\um_0(\zt)-2\bigg],\nonumber\\
    \sigma_2(z,\zt)=&\frac{\partial^2}{\partial z\partial \zt}\bigg[c_ng_p \left(\frac{-1}{\um_0 (z)},\frac{-1}{\um_0(\zt)}\right)+\frac{\zeta_p}{c_n}(1+z\um_0(z))(1+\zt\um_0(\zt))\\
                    &-(1+z\um_0(z))h_p\left(\frac{-1}{\um_0(\zt)}\right)-(1+\zt\um_0(\zt))h_p\left(\frac{-1}{\um_0(z)}\right)\bigg].
  \end{align*} 
  The contours $\mathcal
  C_1$ and $\mathcal C_2$ are non-overlapping, closed, counter-clockwise orientated in the complex plane and  enclosing  the interval $I_c$.
\end{theorem}

\begin{remark}
    Theorem \ref{clt} approximates the distribution of $\Y_n$ by that of a
    Gaussian random vector $\bxi_n$. However, this approximating vector $\bxi_n$ may
    not converge in distribution, that is, the sequence of $\{\big(
    \zeta_p, h_p(z), g_p(z,\tilde z)\big)\}$  which determines the mean and covariance
    functions of $\bxi_n$ may not have a limit as $(p, n)\to \infty$.
  In addition, the convergence of $H_p$ in Assumption (b) is not required and the convergence of $c_n$ in Assumption (a) can be
  weaken to $0 < \liminf c_n \le \limsup c_n <\infty$.
 The proof of this theorem is presented in Section \ref{sec:proofs}.
\end{remark}

\begin{remark}\label{rmkCorr}
Theorem~\ref{clt} contains the  CLT for LSSs of high dimensional
correlation matrices when the population mean is assumed known \citep{Gao2017high}. To see this, consider the simplest case that $\m={\bf 0}$, $w_j\equiv1$  and $\A=\I_p$ in \eqref{eds}, then the sample SSCM under study can be written as $$\B_n=\frac pn \sum_{j=1}^n \frac{\z_j}{\|\z_j\|} \frac{\z_j'}{\|\z_j\|}
=\frac pn \left(\frac{\z_1}{\|\z_1\|}, \cdots, \frac{\z_n}{\|\z_n\|}\right)
 \left(\frac{\z_1}{\|\z_1\|}, \cdots, \frac{\z_n}{\|\z_n\|}\right)'.$$
Denote its companion matrix as 
\begin{align}\label{unb}
\underline{\B}_n=\frac pn \left(\frac{\z_1}{\|\z_1\|}, \cdots, \frac{\z_n}{\|\z_n\|}\right)'\left(\frac{\z_1}{\|\z_1\|}, \cdots, \frac{\z_n}{\|\z_n\|}\right),
\end{align}
which shares the same non-zero eigenvalues as $\B_n$. Thus the result in Theorem \ref{clt} gives the CLT for LSSs  of $\underline{\B}_n$.
Now let's  denote the data matrix as  $\Z=(\z_1, \ldots, \z_n)=(\bv_1, \ldots, \bv_p)'$, where $\z_j$ is the $j$-th column ($j$-th observation)  and $\bv'_j$ is the $j$-th row ($j$-th coordinate) of $\Z$.
Moreover,  the table $\Z$ consists of independent and identically
distributed entries across both the rows and columns so  permuting the
entries in $\Z$ will not change its distribution. 
  The correlation  matrix  $\mathbf{R}_n$ associated with the data set
  $\Z$ can be expressed as
\begin{align}
\mathbf{R}_n=\left(\frac{\bv_1}{\|\bv_1\|}, \cdots, \frac{\bv_p}{\|\bv_p\|}\right)'\left(\frac{\bv_1}{\|\bv_1\|}, \cdots, \frac{\bv_p}{\|\bv_p\|}\right),
\end{align}
which has the same structure (up to a constant factor) as $\underline{\B}_n$ in \eqref{unb} by interchanging the roles of $p$ and $n$.
Therefore in the case of $\A=\I_p$,
the CLT for LSSs of $\mathbf{R}_n$ is readily derived from an
application of 
Theorem~\ref{clt} to the matrix $\underline{\B}_n$.
\end{remark}

\subsection{Example} 
\label{sec:example}

As an illustration, we exhibit the CLT for a widely used LSS which is
the second moment of the eigenvalues of $\B_{n}$, denoted by
$$
\hat\beta_2=\frac{1}{p}\tr (\B_n^2).
$$
We consider the case where the population shape matrix $\A$ is diagonal. In this case, the 
matrix $\T$ in \eqref{tre} can be simplified as
$$\T=\A-\frac{\tau-1}{p}\A^2+\frac{\tau-1}{p^2}\tr\A^2\cdot \A$$
and the three 
 auxiliary quantities in \eqref{hg} become
\begin{align*}
  \zeta_p= \int t^2 d\tilde H_p(t),\quad
  h_p(z)= \int \frac{t^2}{t-z}d\tilde H_p(t),\quad
  g_p(z,\tilde z)= \int \frac{t^2}{(t-z)(t-\tilde z)}d\tilde H_p(t),
\end{align*}
which are only functions of the eigenvalues of $\A$.
 Let $$\alpha_{k,p}=\frac1p\tr(\T^k)=\int t^k d\tilde{H}_p(t).$$
By the relations in \eqref{mp0} and \eqref{um0}, the centering term for the statistic $\hat\beta_2$ is
$$\beta_{2,p}\triangleq \int x^2dF^{c_n, \tilde H_p}(x)= {\alpha}_{2,p}+c_n.$$
The approximating mean and covariance of $p[\hat\beta_2-\beta_{2,p}]$ can be figured out through the residue theorem. For illustration, we calculate the integral corresponding to the
first term in $\mu_1(z)$, that is, 
  \begin{align} \label{in1}
    &-\frac{1}{2\pi\rm i}\oint_{\mathcal C_1} z^2 \int\frac{c_n(\um_0'(z)t)^2d \tilde H_p(t)}{\um_0(z)(1+\um_0(z) t)^3} dz.
  \end{align}
Taking derivatives with respect to $z$ on both sides of \eqref{mp0}, we obtain 
\begin{align*}
\um'_0(z)=\left(\frac{1}{\um^2_0(z)}-c_n\int\frac{t^2}{(1+t\um_0(z))^2}d\tilde H_p(t)\right)^{-1}.
\end{align*}
It then follows that 
\begin{align*}
 \eqref{in1}&=\int-\frac{1}{2\pi\rm i}\oint_{\mathcal C_1} z^2 \frac{c_n \um_0'(z)t ^2}{\um_0(z)(1+\um_0(z) t)^3} d\um_0(z) d \tilde H_p(t)\\
 &=\int-\frac{1}{2\pi\rm i}\oint_{\mathcal C_1}  \frac{c_n t ^2 \Big(z\um_0(z)\Big)^2}{\um_0(z)(1+\um_0(z) t)^3} \left(1-c_n\int\frac{u^2\um^2_0(z)}{(1+u\um_0(z))^2}d\tilde H_p(u)\right)^{-1} d\um_0(z) d \tilde H_p(t)\\
 &=-\int \left\{ \frac{c_n t ^2 \Big(z\um_0(z)\Big)^2}{(1+\um_0(z) t)^3} \left(1-c_n\int\frac{u^2\um^2_0(z)}{(1+u\um_0(z))^2}d\tilde H_p(u)\right)^{-1} \bigg|_{\um_0(z)=0}\right\} d \tilde H_p(t)\\
 &=-c_n \alpha_{2,p}.
\end{align*}
Similar procedure can be repeated to find
 the values of  the remaining contour integrals.
 As a result and by Theorem~\ref{clt},  the  distribution
 of $p [\hat \beta_{2}-  \beta_{2,p}]$ is asymptotically equivalent to 
 the Gaussian distribution   $N(\mu,
 \sigma^2)$, where the mean and variance  parameters are
 given by 
\begin{align*}
\mu &=-c_n{\alpha_{2,p}},\\
\sigma^2& = 8c_n ({\alpha^3_{2,p}} - 2 {\alpha_{2,p}} \alpha_{3,p} + \alpha_{4,p})+4 c_n^2 {\alpha_{2,p}^2}+4c_n(\tau-3)({\alpha^3_{2,p}} - 2 {\alpha_{2,p}} \alpha_{3,p} + \alpha_{4,p}).
\end{align*}

\section{Proofs of the main results}

This section presents the proofs of  Theorem \ref{lsd} and Theorem \ref{clt}.
In all the proofs,  we assume the location vector $\m={\bf 0}$, otherwise, it can be directly subtracted from the sample $\{\bx_j\}$.
We will denote by $K$ some constants appearing in inequalities that can vary from place to place.

\subsection{Proof of Theorem \ref{lsd}}\label{sec:proofs0} 

Let $g_j=p/(\z_j'\A \z_j)$ for $j=1,\ldots,n$, and denote
\begin{align*}
\Z=( z_{ij}),\quad \G={\rm diag}(g_1,\ldots,g_n),\quad \B_n=\frac 1n \A^{\frac 12}\Z\G\Z'\A^{\frac 12},\quad \bS_n=\frac 1n \A^{\frac 12}\Z\Z'\A^{\frac 12}.
\end{align*}
Under Assumptions (a)-(c), the generalized MP law holds true for the sample covariance matrix $\bS_n$ \citep{S95}. Thus it's sufficient to show
\begin{align}\label{nbs}
||\B_n-\bS_n||\xrightarrow{a.s.}0.
\end{align}
To this end,
with the moment conditions in Assumption (c), we shall truncate the variables $(z_{ij})$ at $n^{2/\gamma}$ for some $\gamma\in (4, 4+\delta]$.
Some relevant quantities are denoted as below. For $i=1,\ldots,p$ and $j=1,\ldots,n,$
\begin{align*}
&\hat z_{ij}=z_{ij}I(|z_{ij}|^\gamma\leq n^{2}),\quad\hat \z_j=(\hat z_{1j},\ldots,\hat z_{pj})',\quad \hat g_j=p/(\hat\z_j'\A \hat\z_j),\\
&
\widehat\Z=(\hat z_{ij}),\quad \widehat\G={\rm diag}(\hat g_1,\ldots, \hat g_n),\quad
\widehat \B_n=\frac 1n \A^{\frac 12}\widehat\Z\widehat\G\widehat\Z'\A^{\frac 12},\quad
\widehat \bS_n=\frac 1n \A^{\frac 12}\widehat\Z\widehat\Z'\A^{\frac 12}.
\end{align*}
Note that for the truncated variables $(\hat z_{ij})$, the following results hold automatically
\begin{align}\label{trun1}
\begin{array}{l}
\left|\mathbb{E} \hat z_{ij}\right|=o\left(n^{-2+2/\gamma}\right), \ 
\mathbb{E}\left(\hat z_{ij}^{2}\right)=1+o\left(n^{-2+4/\gamma}\right), \
\mathbb{E}\left(\hat z_{ij}^{4}\right)=\tau+o(1), \
\mathbb{E}|\hat z_{ij}|^{\gamma}<\infty,\
 |\hat z_{ij}|^\gamma<n^2,
\end{array}
\end{align}
and
\begin{align}\label{trun2}
\sum_{k} 2^{k} \E\left|z_{ij}\right|^{\gamma/2}I\big(|z_{ij}| > 2^{2k / \gamma} \big)<\infty.
\end{align}
From \eqref{trun2}  and similar arguments as in the proof of  Lemma 5.12 in \cite{BSbook}, we  have
\begin{align}\label{hatb-b}
{\rm P}(\widehat \B_n\neq \B_n,\ {\rm i.o.})={\rm P}(\widehat \bS_n\neq \bS_n,\ {\rm i.o.})=0.
\end{align}
Next we will prove that for any $\varepsilon>0$ and $k\geq 2$,
\begin{align}\label{pb-s}
{\rm P}\left(||\widehat\B_{n}-\widehat\bS_{n}||>\varepsilon\right)\leq K\varepsilon^{-k}\left(n^{-\frac k2+1}+n^{-\frac{k(\gamma-4)}{\gamma}}\right).
\end{align}
Notice that the spectral norm of the difference between $\widehat\B_{n}$ and $\widehat\bS_{n}$ can be bounded by
\begin{align}\label{b-s}
\left\|\widehat\B_{n}-\widehat\bS_{n}\right\|&\leq ||\A||\frac{|| \widehat\Z\widehat\Z'||}{n}\max_{1\leq j\leq n}\left|\hat g_j-1\right|.
\end{align}
From \cite{BS98},  almost surely, the spectral norm $\|\widehat\Z\widehat\Z'\|/n$ is bounded for all large $n$. Thus, we only need to control the convergence rate of  
$
\max_j\left|\hat g_j-1\right|
$
or 
$
\max_j\left|1/\hat g_j-1\right|.
$ 
By Markov's inequality, for any $\varepsilon>0$ and $k\geq 2$, we have
\begin{align}\label{pdj}
{\rm P}\left(\max_j\left|\frac{1}{\hat g_j}-1\right|>\varepsilon\right)
\leq &np^{-k}\varepsilon^{-k}\E\left|\hat\z_1'\A\hat\z_1-p\right|^{k}.
\end{align}
To bound the expectation in \eqref{pdj}, we divide it into three parts
\begin{align*}
     \E|\hat\z_1'\A\hat\z_1-p|^{k}
\leq &K\E\left|\hat\z_1'\A\hat\z_1-\tilde\z_1'\A\tilde\z_1\right|^{k}
+K\E\left|\tilde\z_1'\A\tilde\z_1-\E\tilde\z_1'\A\tilde\z_1\right|^{k}
+K\left|\E\tilde\z_1'\A\tilde\z_1-p\right|^{k}.
\end{align*}
From \eqref{trun1} and the boundedness of $||\A||$, the first term can be controlled by
\begin{align}\label{eq210}
\E\left|\hat\z_1'\A\hat\z_1-\tilde\z_1'\A\tilde\z_1\right|^{k}
\leq  &K\E\left|\tilde \z_1'\A\E\hat\z_1\right|^k+K \left|\E\hat \z_1'\A\E\hat\z_1\right|^k\nonumber\\
\leq  &K\E^{\frac12}\left|\tilde \z_1'\tilde\z_1\right|^{k}\left|\E\hat \z_1'\E\hat\z_1\right|^{\frac k2}+Kn^{(-3+4/\gamma)k}\nonumber\\
\leq  &K\left[\E\left|\tilde \z_1'\tilde\z_1-\E\tilde \z_1'\tilde\z_1\right|^{k}+\left|\E\tilde \z_1'\tilde\z_1\right|^{k}\right]^{\frac12}n^{(-3/2+2/\gamma)k}+Kn^{(-3+4/\gamma)k}\nonumber\\
\leq  &K\E^{\frac12}\left|\tilde \z_1'\tilde\z_1-\E\tilde \z_1'\tilde\z_1\right|^{k}n^{(-3/2+2/\gamma)k}+Kn^{(-3+4/\gamma)k}\nonumber\\
\leq &Kn^{-1/2+(-3/2+4/\gamma)k}+Kn^{(-1+2/\gamma)k},
\end{align}
where the last inequality is from Lemma \ref{bs98}. Again from this lemma, the second term is bounded by
\begin{align}\label{eq211}
\E\left|\tilde\z_1'\A\tilde\z_1-\E\tilde\z_1'\A\tilde\z_1\right|^{k}\leq K\left(n^{ k/2}+n\E|\tilde z_{11}|^{2k}\right)\leq K\left(n^{k/2}+n^{-1+4k/\gamma}\right).
\end{align}
For the third one, we have from \eqref{trun1}
\begin{align}\label{eq212}
\left|\E\tilde\z_1'\A\tilde\z_1-p\right|^{k}=p^k\left|{\rm Var}(\hat\z_{11})-1\right|^k\leq Kn^{(-1+4/\gamma)k}.
\end{align}
Collecting the results in \eqref{eq210}-\eqref{eq212} yields
\begin{align}\label{eq213}
\E|\hat\z_1'\A\hat\z_1-p|^{k}\leq K\left(n^{k/2}+n^{-1+4k/\gamma}\right),
\end{align}
which together with \eqref{b-s} and \eqref{pdj} give the result in \eqref{pb-s}.
Hence, the conclusion of the theorem follows from \eqref{nbs}, \eqref{hatb-b}, and \eqref{pb-s} with some large $k$.

\subsection{Proof of Theorem \ref{clt} }\label{sec:proofs}

\subsubsection{ Sketch of the proof}

Following the truncation step in the proof of Theorem \ref{lsd}, we now centralize the truncated variables $(\hat z_{ij})$. Some quantities are denoted as below.
\begin{align*}
&\tilde z_{ij}=\hat z_{ij}-\E(\hat z_{ij}),\quad \tilde \z_j=(\tilde z_{1j},\ldots, \tilde z_{pj})',\quad \tilde g_j=p/(\tilde\z_j'\A \tilde\z_j),\quad \widetilde\G={\rm diag}(\tilde g_1,\ldots,\tilde g_n),\\
&\widetilde\Z=(\tilde z_{ij}),\quad \widetilde \B_n=\frac 1n \A^{\frac 12}\widetilde\Z\widetilde\G\widetilde\Z'\A^{\frac 12},\quad
\overline \B_n=\frac 1n\A^{\frac 12}\widetilde\Z\widehat\G\widetilde\Z'\A^{\frac 12},
\quad
\widetilde\bS_n=\frac 1n\A^{\frac 12}\widetilde\Z\widetilde\Z'\A^{\frac 12}.
\end{align*}
Similar to the derivation of \eqref{pb-s}, one may show that
\begin{align}\label{pb-s1}
\max\left\{ {\rm P}\left(||\widetilde\B_{n}-\widetilde\bS_{n}||>\varepsilon\right),
{\rm P}\left(||\overline\B_{n}-\widetilde\bS_{n}||>\varepsilon\right)\right\}
\leq K\varepsilon^{-k}\left(n^{-\frac k2+1}+n^{-\frac{k(\gamma-4)}{\gamma}}\right).
\end{align}
It thus follows from \cite{BS98} that, almost surely, $\lim\sup_n ||\widehat\B_n||$, $\lim\sup_n ||\widetilde\B_n||$, and $\lim\sup_n ||\overline\B_n||$ are all bounded. 

Let $F^{\B_n}$, $F^{\widehat \B_n}$, $F^{\overline \B_n}$, and $F^{\widetilde \B_n}$ be the ESDs of the matrices $\B_n$, $\widehat\B_n$, $\overline\B_n$, and $\widetilde \B_n$, respectively.
Then, for each function $f_j(x)$, we have from \eqref{hatb-b}
\begin{align}\label{eq220}
p\left|\int f_j(x)dF^{\B_n}- \int f_j(x)dF^{\widehat\B_n}\right|\xrightarrow{a.s.}0.
\end{align}
By Corollary A.37 in \cite{BSbook}, it holds that
\begin{align}\label{eq222}
p\left|\int f_j(x)dF^{\widehat\B_n}- \int f_j(x)dF^{\overline\B_n}\right|
\leq & K_j\sum_{k=1}^p\left|\lambda_k^{\widehat\B_n}-\lambda_k^{\overline\B_n}\right|\\
\leq &2 K_{j}\left[c_n \tr \A^{\frac 12}\left(\widehat\Z-\widetilde\Z\right)\widehat\G\left(\widehat\Z-\widetilde\Z\right)'\A^{\frac 12}\left(||\widehat\B_n||+||\overline\B_n||\right)\right]^{1 / 2}.\nonumber
\end{align}
where $K_j$ is an upper bound on $|f_j'(x)|$ and $\lambda_k^{\B}$ denotes the $k$-th largest eigenvalue of the matrix $\B$. By \eqref{trun1} and \eqref{pdj}, one may get 
\begin{align*}
\left|\tr \A^{\frac 12}\left(\widehat\Z-\widetilde\Z\right)\widehat\G\left(\widehat\Z-\widetilde\Z\right)'\A^{\frac 12}\right|\leq ||\A||\max_j|\hat g_j|\tr(\E\widehat\Z\E\widehat\Z')\xrightarrow{a.s.}0,
\end{align*}
and thus \eqref{eq222} is $o_{a.s.}(1)$.
Moreover, from \eqref{pdj} and \eqref{eq210}, applying Markov's inequality, we have also 
\begin{align}\label{eq223}
p\left|\int f_j(x)dF^{\overline\B_n}- \int f_j(x)dF^{\widetilde\B_n}\right|
\leq &K_jp||\overline\B_n-\widetilde\B_n||\nonumber\\
\leq &K_jp||\A||\cdot\left\|\frac1n \widetilde\Z\widetilde\Z'\right\|\max_j|\hat g_j-\tilde g_j|\xrightarrow{a.s.}0.
\end{align}
Collecting \eqref{eq220}, \eqref{eq222}, and \eqref{eq223}, we get
\begin{align}\label{eq221}
p\left|\int f_j(x)dF^{\B_n}- \int f_j(x)dF^{\widetilde\B_n}\right|\xrightarrow{a.s.}0.
\end{align}
 Therefore, it is sufficient to prove the theorem by replacing the matrix $\B_n$ with its truncated and centralized version $\widetilde \B_n$, or equivalently, we assume
\begin{align}\label{tcr}
\E(z_{11})=0,\quad \E(z_{11}^2)=1,\quad \E(z_{11}^4)=\tau+o(1),\quad \E(|z_{11}|^\gamma)<\infty,\quad \max_{i,j}|z_{ij}|^\gamma<n^2,
\end{align}
with $\gamma=5$ for the proof of the theorem.
Note that we assume $\E(z_{11}^2)=1$ without rescaling $(\tilde z_{ij})$ since the sample spatial-sign vectors $\{\A^{1/2}\z_j/\|\A^{1/2}\z_j\|\}$ are all self-normalized.

Next we define a rectangular contour enclosing the interval $I_c=[s_l, s_r]$,
\begin{equation}\label{slr}
s_l=\liminf_{p\rightarrow\infty}\lambda_{\min}^{\T}(1-\sqrt{c})^2I_{(0,1)}(c)\quad  \text{and}\quad s_r=\limsup_{p\rightarrow\infty}\lambda_{\max}^{\T}(1+\sqrt{c})^2,	
\end{equation}
and thus enclosing all supports of the LSDs $\{F^{c_n,\tilde H_p}\}$. 
Choosing two numbers $x_l<x_r$ such that $[s_l, s_r]\subset (x_l,x_r)$
and letting $v_0 > 0$ be arbitrary,  then the contour can be described as 
\begin{align*}
\mathcal C= \{x\pm iv_0: x\in[x_l,x_r]\}\cup \{x+iv: x\in \{x_r,x_l\}, v\in[-v_0,v_0]\}. 
\end{align*} 
Denote
\begin{align*}
m_n(z)&=\int \frac{1}{x-z}dF^{\B_n}(x),\quad m_0(z)=\int \frac{1}{x-z}dF^{c_n,\tilde H_p}(x),\quad \um_n(z)=	-\frac{1-c_n}{z}+c_nm_n(z),
\end{align*}
We then  
define a random process on $\mathcal C$ as 
$$ 
M_{n}(z)=p[m_{n}(z)-m_{0}(z)]=n[\um_{n}(z)-\um_{0}(z)],\quad z\in\mathcal C, 
$$
where $\um_0(z)$ is defined in \eqref{um0}. 
From Cauchy's integral formula, for any $k$ analytic functions $(f_\ell)$ and complex numbers $(a_\ell)$, we have
\begin{align*}
	\sum_{\ell= 1}^kp a_\ell\int f_\ell(x)dG_n(x)=-\sum_{\ell= 1}^k\frac{ a_\ell}{2\pi \rm i}\oint_{\mathcal C}f_\ell(z) M_n(z)dz
\end{align*}
when all sample eigenvalues fall in the interval $(x_l, x_r)$, which holds with probability $1-o(n^{-s})$ for any $s>0$. That is,
	\begin{align}\label{hprob}
	{\rm P}(||\B_n||>x_r)=o(n^{-s})\quad\text{and}\quad 	{\rm P}(\lambda_{\min}^{\B_n}<x_l)=o(n^{-s}),\ \forall s>0,
	\end{align}
which follows from \eqref{pb-s1} and a similar conclusion for $\bS_n$ \citep{BS04}.
 In order to deal with the small probability event where some eigenvalues are outside  the interval $(x_l, x_r)$ in finite dimensional situations,
\cite{BS04} suggested truncating $M_n(z)$ as,  for $z=x+{\rm i}v\in\mathcal C$,  
\begin{align*} 
\widehat M_{n}(z)= 
\begin{cases} 
M_{n}(z)&z\in \mathcal C_{n},\\ 
M_{n}(x+{\rm i}n^{-1}\varepsilon_n)& x\in \{x_l,x_r\}\quad\text{and}\quad v\in[0, n^{-1}\varepsilon_n],\\ 
M_{n}(x-{\rm i}n^{-1}\varepsilon_n)& x\in \{x_l,x_r\}\quad\text{and}\quad v\in[-n^{-1}\varepsilon_n, 0], 
\end{cases} 
\end{align*} 
where $ \mathcal C_{n}= \{x\pm {\rm i}v_0: x\in[x_l,x_r]\}\cup \{x\pm {\rm i}v: x\in\{x_l, x_r\}, v\in[n^{-1}\varepsilon_n, v_0]\},$
a regularized version of $\mathcal C$ excluding a small segment near the real line,
and the positive sequence $(\varepsilon_n)$ decreases to zero satisfying  $\varepsilon_n>n^{-a}$ for some $a\in(0,1)$. From this and \eqref{hprob}, one may thus find
\begin{align*}
	\oint_{\mathcal C}f_\ell (z)M_n(z)dz=\oint_{\mathcal C} f_\ell(z)\widehat M_n(z)dz+o_p(1),
	\end{align*}
for every $\ell\in \{1,\ldots, k\}$.
Hence, the proof of Theorem \ref{clt} can be completed by verifying the convergence of $\widehat M_n(z)$ on $\mathcal C$ as stated in the following lemma.

\begin{lemma}\label{clt-mn} 
In addition to Assumptions (a)-(c), suppose that the conditions in \eqref{tcr} hold with $\gamma=5$. We have
	$$
	\widehat M_n(z)\stackrel{d}{=}M_0(z)+o_p(1),\quad z\in \mathcal C,
	$$
	where the random process $M_0(z)$ is a two-dimensional Gaussian process. The mean function is 
	\begin{align} 
	{\E}M_0(z) 
	=\mu_1(z)+(\tau-3)\mu_2(z),
	\label{mean} 
	\end{align} 
and the covariance function is
	\begin{align*} 
	\Cov(M_0(z),M_0(\zt))=\sigma_1(z,\zt)+(\tau-3)\sigma_2(z,\zt).
	\end{align*}
\end{lemma}

\subsubsection{Proof of Lemma \ref{clt-mn}}
	
Some quantities are listed below which will be used frequently throughout  this proof. 
\begin{align*} 
&\s_{j}=\s(\bx_j),\quad \br_j=\sqrt{p/n}\s_{j},\quad \bSigma=n\E \br_1\br_1',\\
&\D(z)=\B_{n}-zI,\quad 
\D_{j}(z)=\D (z)-\br_j\br'_j,\quad
\D_{ij}(z)=\D (z)-\br_i\br_i'-\br_j\br'_j,\ (i\neq j),\\ 
&\varepsilon_{j}(z)=\br_j'\D_{j}^{-1}(z)\br_j-\frac{1}{n}\tr\bSigma\D_{j}^{-1}(z),\quad
\gamma_{j}(z)=\br_j'\D_{j}^{-1}(z)\br_j-\frac{1}{n}\E\tr\bSigma\D_{j}^{-1}(z),\\
&\delta_{j}(z)=\br_j'\D_{j}^{-2}(z)\br_j-\frac{1}{n}\tr\bSigma \D_{j}^{-2}(z),\\ 
&\beta_{j}(z)=\frac{1}{1+\br_j'\D_{j}^{-1}(z)\br_j},\quad \beta_{j k}(z)=\frac{1}{1+\br_j'\D_{kj}^{-1}(z)\br_j},\\
&\bar \beta_{j}(z)=\frac{1}{1+n^{-1}\tr\bSigma \D_{j}^{-1}(z)}, \quad \bar \beta_{j k}(z)=\frac{1}{1+n^{-1}\tr\bSigma \D_{k j}^{-1}(z)},\\ 
& b_{n}(z)=\frac{1}{1+n^{-1}\E\tr\bSigma \D_{j}^{-1}(z)}, \quad \bar b_{n}(z)=\frac{1}{1+n^{-1}\E\tr\bSigma \D_{kj}^{-1}(z)}. 
\end{align*}
Note that by Lemma \ref{st}, the matrix $\bSigma$ and $\T$ defined in \eqref{tcr} are asymptotically equivalent and the last six quantities are bounded in absolute value by $|z|/v$ for any $z=u+iv\in \mathbb C^+$.
Now we split $\widehat M_{n}(z)$ into two parts as  
	\begin{align*}
	\widehat M_{n}(z) = & p[m_{n} (z) - \E m_{n} (z)]+p[\E m_{n}(z)-m_0(z)]\\
	:=&M_{n}^{(1)}(z)+ M_{n}^{(2)}(z).
	\end{align*}
	Hence, the convergence of $\widehat M_{n}(z)$ can be obtained through the following three steps.
	\begin{itemize}
		\item[]{\bf Step 1}: Finite dimensional convergence of $M_n^{(1)}(z)$.
		Let $z_1,\ldots,z_q$ be any $q$ complex numbers on $\mathcal C_n$, this step approximates joint distribution of 
		\begin{align}\label{rv}
		\big[M^{(1)}_n(z_1),\ldots, M^{(1)}_n(z_q)\big]
		\end{align} 
		through martingale CLT \citep{B95}. Beyond the techniques used in  \cite{BS04}, a particularly important problem is to find new approaches to deal with the non-linear correlation structure among the entries of $\s(\bx_j)$.  And such non-linear correlation is actually introduced by the spatial-sign transform of the data, to be precise, the norm $\|\bx_j\|$ that appears in the denominator of $\s(\bx_j)$. To this end, by giving an asymptotic expansion of  $\s(\bx_j)$, we develop Lemma \ref{double-e} concerning the covariance of certain quadratic forms, which turns out to be one of the cornerstones for establishing our new CLT.
		

		\item[] {\bf Step 2}: Tightness of $M_n^{(1)}(z)$  on $\mathcal C_n$. We illustrate in this step the basic idea for proving the tightness. As shown in \eqref{hprob}, the probability of extreme eigenvalues falling outside the contour $\mathcal C$ can be well controlled. By virtue of this and Lemma \ref{ineq}, the tightness can be obtained following similar arguments in \cite{BS04}.
		
		\item[] {\bf Step 3}: Convergence of $M_n^{(2)}(z)$. In this final step, we approximate the quantity $M_n^{(2)}(z)$. In parallel with {\bf Step 1}, dealing with the nonlinear effects as shown in Lemma \ref{double-e} is the main focus in this part.
		As will be seen, such nonlinear effects will contribute several new terms to the mean of $\bxi_n$.
	\end{itemize}

\medskip

\textit{\bf Step 1: Finite dimensional convergence of $M_{n}^{(1)}(z)$ in distribution.}  
\smallskip

 Let $\E_0(\cdot)$ denote expectation and $\E_j(\cdot)$ denote conditional expectation with respect to the $\sigma$-field generated by $\br_1,\ldots,\br_j$, $j=1,\ldots,n$. 
From the martingale decomposition and the identity  
\begin{align}\label{D-D} 
\D ^{-1}(z)-\D_{j}^{-1}(z)=-\D_{j}^{-1}(z)\br_j\br_j'\D_{j}^{-1}(z)\beta_{j}(z), 
\end{align} 
we get 
\begin{align} 
M_{n}^{(1)}(z)
&=\sum_{j=1}^n(\E_j-\E_{j-1})\tr\left[\D ^{-1}(z)-\D_{j}^{-1}(z)\right]\nonumber\\
&=\sum_{j=1}^n(\E_j-\E_{j-1})\frac{d\log(\beta_j(z)/\bar{\beta}_j(z))}{d z},\nonumber\\
&=\frac{d}{d z}\sum_{j=1}^n(\E_j-\E_{j-1})\log[1-\bar\beta_j(z)\varepsilon_j(z)+\bar\beta_j(z)\beta_j(z)\varepsilon_j^2(z)],\label{md} 
\end{align} 
where the last equality is from the identity $\beta_{j}(z)=\bar \beta_{j}(z)-\bar\beta_{j}^2(z)\varepsilon_{j}(z)+\bar\beta_{j}^2(z)\beta_{j}(z)\varepsilon_{j}^2(z)$.
From Lemma \ref{ineq} and the boundedness of $\beta_j(z)$ and $\bar\beta_j(z)$, 
we   have
$$
	\E\bigg|\sum_{j=1}^n(\E_j-\E_{j-1})\bar\beta_j(z)\beta_j(z)\varepsilon^2_j(z)\bigg|^2\leq Kn	\E\left|\varepsilon_j(z)\right|^4	\to0.
$$
Thus applying Taylor's expansion to the log function in \eqref{md}, one may conclude
\begin{align*}
M_n^{(1)}(z)&=-\frac{d}{dz}\sum_{j=1}^n (\E_j-\E_{j-1})\bar\beta_j(z)\varepsilon_j(z)+o_p(1)\\
&=-\frac{d}{dz}\sum_{j=1}^n \E_j\bar\beta_j(z)\varepsilon_j(z)+o_p(1).
\end{align*}
Therefore, we turn to consider the martingale difference sequence  
$$ 
Y_{nj}(z):=\frac{d}{dz}\E_j\bar\beta_{j}(z)\varepsilon_{j}(z),\ j=1,\ldots,n.
$$ 
The Lyapunov condition for this sequence is guaranteed by the fact that 
\begin{align*} 
\sum_{j=1}^n \E\left|Y_{nj}(z)\right|^4
&=\sum_{j=1}^n \E\left|\E_j\left(\delta_j(z)\bar\beta_j(z)-\varepsilon_j(z)\bar\beta_j^2(z)\frac{1}{n}\tr\bSigma \D_j^{-2}(z)\right)\right|^4\\ 
&	\leq  K\sum_{j=1}^n 
	\left(\frac{|z|^4\E|\delta_{j}(z)|^4}{v^4}+\frac{|z|^8p^4\E|\varepsilon_{j}(z)|^4}{v^{16}n^4}\right)\to0, 
\end{align*} 
where the convergence is from Lemma \ref{ineq}.

We next consider the sum
$
\sigma_n(z, \zt)\triangleq \sum_{j=1}^n \E_{j-1}\left[Y_{nj}(z)Y_{nj}(\zt)\right],
$
for $z\neq \zt\in\{z_1,\ldots,z_w\}$.
Notice that
\begin{equation}\label{bnz} 
\E|\bar\beta_j(z)-b_n(z)|\leq \frac{K}{n}\E\left|\tr\bSigma \D_j^{-1}(z)-\E\tr\bSigma \D_j^{-1}\right|\to0 \quad\text{and}\quad  b_{n}(z)+z\um_0(z)\rightarrow 0, 
\end{equation}
which follow from \cite{BS04} and Lemma \ref{st}, and thus we have
\begin{equation*}\label{cov-term-2} 
\sigma_n(z, \zt)=\frac{\partial^2}{\partial z\partial \tilde z}z\zt\um_0(z)\um_0(\zt)\sum_{j=1}^n\E_{j-1}\left(\E_j\varepsilon_j(z)\E_j\varepsilon_j(\zt)\right) +o_p(1).
\end{equation*} 
Moreover, applying Lemma \ref{double-e} to the above conditional expectations, one may get  
\begin{align*}
&z\zt\um_0(z)\um_0(\zt)\sum_{j=1}^n\E_{j-1}\left(\E_j\varepsilon_j(z)\E_j\varepsilon_j(\zt)\right)\\
=& 2T_1+ \frac 2p\tr(\T^2)T_2-2T_3-2T_4+(\tau-3)(T_5+T_6-T_7-T_8)+o(1), 
\end{align*} 
where  
\begin{align*} 
	T_1&=\frac{z\zt\um_0(z)\um_0(\zt)}{n^2}\sum_{j=1}^n\tr\left[\E_j\T \D_{j}^{-1}(z)\E_j\T \D_{j}^{-1}(\zt)\right],\\ 
	T_2&=\frac{z\zt\um_0(z)\um_0(\zt)}{pn^2}\sum_{j=1}^n\tr\left[\E_j\T \D_{j}^{-1}(z)\right]\tr\left[\E_j\T \D_{j}^{-1}(\zt)\right],\\ 
	T_3&=\frac{z\zt\um_0(z)\um_0(\zt)}{pn^2}\sum_{j=1}^n\tr\left[\E_j\T^2 \D_{j}^{-1}(z)\right]\tr\left[\E_j\T \D_{j}^{-1}(\zt)\right],\\ 
	T_4&=\frac{z\zt\um_0(z)\um_0(\zt)}{pn^2}\sum_{j=1}^n\tr\left[\E_j\T \D_{j}^{-1}(z)\right]\tr\left[\E_j\T^2 \D_{j}^{-1}(\zt)\right],\\
	T_5&=\frac{z\zt\um_0(z)\um_0(\zt)}{n^2}\sum_{j=1}^n\tr\left[\E_j(\T^{\frac 12} \D_{j}^{-1}(z)\T^{\frac 12})\circ\E_j(\T^{\frac 12}\D_{j}^{-1}(\zt)\T^{\frac 12})\right], \\
	T_6&=\frac{z\zt\um_0(z)\um_0(\zt)}{p^2n^2}\sum_{j=1}^n\tr\left[\E_j\T \D_{j}^{-1}(z)\right]\tr\left[\E_j\T \D_{j}^{-1}(\tilde z)\right]\tr\left[\T \circ \T\right],\\
	T_7&=\frac{z\zt\um_0(z)\um_0(\zt)}{pn^2}\sum_{j=1}^n\tr\left[\E_j\T \D_{j}^{-1}(z)\right]\tr\left[\E_j(\T^{\frac 12} \D_{j}^{-1}(\tilde z)\T^{\frac 12})\circ \T\right],\\
	T_8&=\frac{z\zt\um_0(z)\um_0(\zt)}{pn^2}\sum_{j=1}^n\tr\left[\E_j(\T^{\frac 12} \D_{j}^{-1}( z)\T^{\frac 12})\circ \T\right]\tr\left[\E_j\T \D_{j}^{-1}(\tilde z)\right].
\end{align*}
Following similar steps as in \cite{BS04} and \cite{H19},  applying Lemma \ref{st} and Lemma \ref{ineq}, we obtain
\begin{align*}
T_1&=\log\frac{\um_0(z)-\um_0(\zt)}{\um_0(z)\um_0(\zt)(z-\zt)}+o_p(1),\\
T_2=\frac{T_6}{\zeta_p}&=c_n\int\frac{t\um_0(z)d \tilde H_p(t)}{1+t\um_0(z)}\int\frac{t\um_0(\zt)d \tilde H_p(t)}{1+t\um_0(\zt)}+o_p(1)\\
&=\frac{[1+z\um_0(z)][1+\zt\um_0(\zt)]}{c_n}+o_p(1).
\end{align*}
Notice that statistics $T_3$ and $T_4$ will reduce to $T_2$ if $\T^2$ is replaced with $\T$. By this, we have 
\begin{align*} 
	T_3&=c_n\int\frac{t^2\um_0(z)d \tilde H_p(t)}{1+t\um_0(z)}\int\frac{t\um_0(\zt)d \tilde H_p(t)}{1+t\um_0(\zt)}+o_p(1)\\
	&=\bigg[1-\frac{1+z\um_0(z)}{c_n\um_0(z)}\bigg][1+\zt\um_0(\zt)]+o_p(1),\\ 
	T_4&=c_n\int\frac{t\um_0(z)d \tilde H_p(t)}{1+t\um_0(z)}\int\frac{t^2\um_0(\zt)d \tilde H_p(t)}{1+t\um_0(\zt)}+o_p(1)\\
	&=\bigg[1-\frac{1+\zt\um_0(\zt)}{c_n\um_0(\zt)}\bigg][1+z\um_0(z)]+o_p(1). 
\end{align*} 
For the terms $T_5$, $T_7$, and $T_8$, following similar procedure as in   \cite{PZ08} for proving their Theorem 1.4, using Lemma \ref{st}, Lemma \ref{ineq}, and Theorem \ref{lsd}, one may get
\begin{align*}
T_5&=\frac{1}{n}\tr\left[(\T^{\frac 12} (\um_0^{-1}(z)\I+\T)^{-1}\T^{\frac 12})\circ(\T^{\frac 12} (\um_0^{-1}(\zt)\I+\T)^{-1}\T^{\frac 12})\right]+o_p(1) \\
&=c_ng_p \left(\frac{-1}{\um_0 (z)},\frac{-1}{\um_0(\zt)}\right)+o_p(1),\\
T_7&=\frac{1}{pn}\tr\left[\T (\um_0^{-1}(z)\I+\T)^{-1}\right]\tr\left[(\T^{\frac 12}(\um_0^{-1}(\zt)\I+\T)^{-1}\T^{\frac 12})\circ \T\right]+o_p(1)\\
&= h_p\left(\frac{-1}{\um_0(\zt)}\right)[1+z\um_0(z)]+o_p(1),\\
T_8&=\frac{1}{pn}\tr\left[(\T^{\frac 12} (\um_0^{-1}(z)\I+\T)^{-1}\T^{\frac 12})\circ \T\right]\tr\left[\T (\um_0^{-1}(\zt)\I+\T)^{-1}\right]+o_p(1)\\
&= h_p\left(\frac{-1}{\um_0(z)}\right)[1+\zt\um_0(\zt)]+o_p(1).
\end{align*}
Collecting the above results we get 
\begin{equation*} 
\eqref{rv}\stackrel{d}{=}\left[M_0^{(1)}(z_1),\ldots,M_0^{(1)}(z_q)\right]+o_p(1), 
\end{equation*} 
where $[M_0^{(1)}(z_1),\ldots, M_0^{(1)}(z_q)]$ is a $q$-dimensional zero-mean Gaussian random vector with covariance function
$$
\Cov[M_0^{(1)}(z),M_0^{(1)}(\zt)]=\sigma_1(z,\zt)+(\tau-3)\sigma_2(z,\zt).
$$

\textbf{Step 2: Tightness of $M_n^{(1)}(z)$.} 
The tightness can be established by verifying the moment condition (12.51) of \cite{B68}: 
\begin{equation}\label{tightness} 
\sup_{n,z_1,z_2\in \mathcal C_n}\frac{\E|M_n^{(1)}(z_1)-M_n^{(1)}(z_2)|^2}{|z_1-z_2|^2}<\infty.
\end{equation} 
By \eqref{hprob} and arguments in \cite{BS04}, the moments of $\D^{-1}(z)$, $\D^{-1}_j(z)$ and $\D^{-1}_{ij}(z)$ are uniformly bounded in $n$ and $z\in \mathcal C_n$, that is, for any positive $k$,   
\begin{equation}\label{D-bound} 
\max \{ \E||\D^{-1}(z)||^k, \E||\D^{-1}_j(z)||^k, \E||\D^{-1}_{ij}(z)||^k\}\leq K. 
\end{equation}
By such boundedness, the inequality in Lemma \ref{ineq} can be extended to 
\begin{align} 
\left|\E\left[a(v)\prod_{l=1}^{k}\left(\br'\B_l(v)\br-\frac1n\tr\bSigma \B_l(v)\right)\right]\right| 
\leq Kn^{-1- k(\gamma-4)/\gamma},\quad k\geq 2.\label{ineq-ex} 
\end{align} 
The matrices $\B_l(v)$ in \eqref{ineq-ex} are independent of $\br$ and 
\begin{equation*} 
\max\{|a(v)|, ||\B_l(v)|| \}\leq K\left[1+p^{s}I\left(||\B_n||\geq x_r \ \text{or}\ \lambda_{\min}^{\tilde \B}\leq x_l\right)\right] 
\end{equation*} 
for some positive $s$, where $\tilde \B$ denotes $\B_n=\sum \br_j\br_j'$, $\B_j=\sum_{k\neq j}\br_k\br_k'$, or $\B_{ij}=\sum_{k\neq i,j}\br_k\br_k'$.
Finally, following similar procedure as in Section 3 of  \cite{BS04}, and applying Lemma \ref{st}, Lemma \ref{ineq} together with \eqref{hprob}, \eqref{D-bound}, and \eqref{ineq-ex},  one may verify \eqref{tightness}. The details are thus omitted.

\textbf{Step 3: Convergence of $M_n^{(2)}(z)$.} 
To finish the proof, it is enough to show that the sequence of $M_n^{(2)}(z)$ is bounded and equicontinuous, and is equal to the mean function \eqref{mean} asymptotically. 
The boundedness and equicontinuity can be verified following the arguments in \cite{BS04}. We next propose 
a novel method to approximate $M_n^{(2)}(z)$, which is quite different from the idea in \cite{BS04}. This new procedure is more straightforward and easier to follow. Before the proof,
we first list some results that will be used in this part: 
\begin{align}
&\sup_{z\in\mathcal C_n}\E|\varepsilon_j(z)|^k\leq Kn^{-1- k(\gamma-4)/\gamma},\quad \sup_{z\in\mathcal C_n}\E|\gamma_j(z)|^k\leq Kn^{-1- k(\gamma-4)/\gamma},\label{varep-bound}\\ 
&\sup_{n,z\in\mathcal C_n}|b_n(z)+z\um_0(z)|\to0,\quad\sup_{n,z\in\mathcal C_n}||z\I- b_{n}(z)\bSigma||^{-1}<\infty,\label{supsup}\\ 
&\sup_{n,z\in\mathcal C_n}\E|\tr \D^{-1}(z)\bM-\E\tr \D^{-1}(z)\bM|^2\leq K||\bM||^2,\label{tr-dm-bound} 
\end{align} 
where $k\geq 2$ and $\bM$ is a nonrandom $p\times p$ matrix. 
These results can be verified step by step following the discussions in \cite{BS04} and we omit the details.

Writing $\bV(z)=zI- b_{n}(z)\bSigma$, we decompose $M_n^{(2)}(z)$ in two ways: 
\begin{align*} 
M_n^{(2)}(z)&=[p\E m_n (z)+\tr \bV^{-1}(z)]-[\tr \bV^{-1}(z)+pm_0(z)]
:= S_{n}(z)-T_n(z),\\ 
M_n^{(2)}(z)&=[n\E\um_n (z)+nb_n(z)/z]-[nb_n(z)/z+n\um_0(z)]
:= \uS_n(z)-\uT_n(z). 
\end{align*} 
Notice that by Lemma \ref{st},
\begin{align*} 
T_n(z)&=p\int \frac{d \tilde H_p(t)}{z- b_{n}(z)t}-p\int \frac{d \tilde H_p(t)}{z+z\um_{0}(z)t}+o(1)\\ 
&=p\left[ b_{n}(z)+z\um_{0}(z)\right]\int \frac{td \tilde H_p(t)}{(z- b_{n}(z)t)(z+z\um_{0}(z)t)}+o(1)\\ 
&=c_n\uT_n(z)\int \frac{td \tilde H_p(t)}{(z- b_{n}(z)t)(1+\um_{0}(z)t)}+o(1). 
\end{align*} 
From this and the convergence in \eqref{supsup}, we have 
\begin{align}\label{tnt} 
\frac{M_n^{(2)}(z)-S_n(z)}{M_n^{(2)}(z)-\uS_n(z)}&=\frac{T_n(z)}{\uT_n(z)}=\frac{c_n}{z}\int \frac{td \tilde H_p(t)}{(1+\um_{0}(z)t)^2}+o(1).
\end{align} 

Our next task is to study the convergence of $S_n(z)$ and $\uS_n(z)$. For simplicity of notation, we suppress the expression $z$ in the sequel when it is served as independent variables of some functions.  
All expressions and convergence statements hold uniformly for $z\in \mathcal C_n$.  

We first simplify the expression of $S_n$. Using the identity $\br_j'\D^{-1}=\br_j'\D_j^{-1}\beta_j$, we have 
\begin{align} 
S_n&=\E\tr(\D^{-1}+\bV^{-1})
=\E\tr\left[\bV^{-1}\left(\sum_{j=1}^n\br_j\br_j'- b_{n}\bSigma\right)\D^{-1}\right]\nonumber\\ 
&=n\E \beta_1\br_1'\D_1^{-1}\bV^{-1}\br_1-b_{n}\E\tr\bSigma \D^{-1}\bV^{-1}.\label{pn} 
\end{align} 
From \eqref{D-D} and $\beta_1=b_n-b_n\beta_1\gamma_1$, 
\begin{align*} 
\E\tr \bV^{-1}\bSigma (\D_1^{-1}-\D^{-1}) 
&=\E\tr \bV^{-1}\bSigma \D_1^{-1}\br_1\br_1'\D_1^{-1}\beta_1\nonumber\\ 
&=b_n\E(1-\beta_1\gamma_1) \br_1'\D_1^{-1}\bV^{-1}\bSigma \D_1^{-1}\br_1, 
\end{align*} 
where $|\E\beta_1\gamma_1\br_1'\D_1^{-1}\bV^{-1}\bSigma \D_1^{-1}\br_1|\leq Kn^{-1/2}$. 
From this and \eqref{pn}, we get 
\begin{align*} 
S_n=n\E \beta_1\br_1'\D_1^{-1}\bV^{-1}\br_1-b_{n}\E\tr\bSigma \D_1^{-1}\bV^{-1}+\frac{1}{n}b_n^2\E\tr \D_1^{-1}\bV^{-1}\bSigma \D_1^{-1}\bSigma+o(1). 
\end{align*} 
Then plugging $\beta_1=b_n-b_n^2\gamma_1+b_n^3\gamma_1^2-\beta_1b_n^3\gamma_1^3$ into the first term in the above equation, we obtain 
\begin{align*} 
n\E \beta_1\br_1'\D_1^{-1}\bV^{-1}\br_1&=b_n\E \tr \D_1^{-1}\bV^{-1}\bSigma-nb_n^2\E\gamma_1\br_1'\D_1^{-1}\bV^{-1}\br_1\\
&+nb_n^3\E\gamma_1^2\br_1'\D_1^{-1}\bV^{-1}\br_1-nb_n^3\E\beta_1\gamma_1^3\br_1'\D_1^{-1}\bV^{-1}\br_1. 
\end{align*} 
Note that, from \eqref{ineq-ex}, \eqref{varep-bound}, and \eqref{tr-dm-bound}, 
\begin{align*} 
\E\gamma_1\br_1'\D_1^{-1}\bV^{-1}\br_1 
&=\E\left[\br_1'\D_1^{-1}\br_1-\frac{1}{n}\tr \D_1^{-1}\bSigma\right]\bigg[\br_1'\D_1^{-1}\bV^{-1}\br_1\\
&-\frac{1}{n}\tr \D_1^{-1}\bV^{-1}\bSigma\bigg]
+\frac{1}{n^2}\Cov(\tr \D_1^{-1}\bSigma, \tr \D_1^{-1}\bV^{-1}\bSigma)\\ 
&=\E\left[\br_1'\D_1^{-1}\br_1-\frac{1}{n}\tr \D_1^{-1}\bSigma\right]\bigg[\br_1'\D_1^{-1}\bV^{-1}\br_1\\
&-\frac{1}{n}\tr \D_1^{-1}\bV^{-1}\bSigma\bigg]
+o\left(\frac{1}{n}\right),\\ 
\E\gamma_1^2\br_1'\D_1^{-1}\bV^{-1}\br_1 
&=\E\gamma_1^2\left[\br_1'\D_1^{-1}\bV^{-1}\br_1-\frac{1}{n}\tr \D_1^{-1}\bV^{-1}\bSigma\right]\\ 
&+\frac{1}{n}\Cov(\gamma_1^2, \tr \D_1^{-1}\bV^{-1}\bSigma)+\frac{1}{n}\E\gamma_1^2E\tr \D_1^{-1}\bV^{-1}\bSigma\\ 
&=\frac{1}{n}\E\gamma_1^2\E\tr \D_1^{-1}\bV^{-1}\bSigma+o\left(\frac{1}{n}\right),\\ 
\E\beta_1\gamma_1^3\br_1'\D_1^{-1}\bV^{-1}\br_1
&=o\left(\frac{1}{n}\right). 
\end{align*} 
We thus arrive at  
\begin{align*} 
S_n 
&=-nb_n^2\E\left[\br_1'\D_1^{-1}\br_1-\frac{1}{n}\tr \D_1^{-1}\bSigma\right]\left[\br_1'\D_1^{-1}\bV^{-1}\br_1-\frac{1}{n}\tr \D_1^{-1}\bV^{-1}\bSigma\right]\nonumber\\ 
&+b_n^3\E\gamma_1^2\E\tr \D_1^{-1}\bV^{-1}\bSigma +\frac{1}{n}b_n^2\E\tr \D_1^{-1}\bV^{-1}\bSigma \D_1^{-1}\bSigma +o(1).\label{spn} 
\end{align*} 
On the other hand, by the identity $\br_j'\D^{-1}=\br_j'\D_j^{-1}\beta_j$, we have 
$$ 
p+z\tr \D^{-1}=\tr(\B_n\D^{-1})=\sum_{j=1}^n\beta_j\br_j'\D_j^{-1}\br_j=n-\sum_{j=1}^n\beta_j, 
$$ 
which implies  $nz\um_n=-\sum_{j=1}^n\beta_j$. From this, together with $\beta_1=b_n-b_n^2\gamma_1+b_n^3\gamma_1^2-\beta_1b_n^3\gamma_1^3$ and \eqref{ineq-ex}, we get 
\begin{eqnarray*}\label{supn} 
	\uS_n=-\frac{n}{z}\E\left(\beta_1- b_n\right) 
	=-\frac{n}{z}b_n^3\E\gamma_1^2+o(1). 
\end{eqnarray*} 
 Applying Lemma \ref{double-e} to the simplified  $S_n$ and $\uS_n$, and then replacing $\D_j$ with $\D$ in the derived results 
yield 
\begin{align*} 
S_n 
&=-\frac{b_n^2}{n}\bigg[\E\tr \D^{-1}\T \D^{-1}\bV^{-1}\T+\frac{2}{p}\bigg( \frac1n\tr\T^2\E\tr\T \D^{-1}\tr\T\D^{-1}\bV^{-1}\\ 
&\quad -\E\tr\T^2\D^{-1}\tr\T\D^{-1}\bV^{-1}-\E\tr\T \D^{-1}\tr \T^2\D^{-1}\bV^{-1}\bigg)\bigg]\nonumber\\ 
&+\frac{2b_n^3}{n^2}\bigg[\E\tr \D^{-1}\T \D^{-1}\T+\frac{1}{p}\bigg( \frac1n\tr\T^2\E\tr\T \D^{-1}\tr\T\D^{-1}-2\E\tr\T^2\D^{-1}\tr\T\D^{-1}\bigg) \bigg] \\ 
&\cdot \E\tr \D^{-1}\bV^{-1}\T-\frac{(\tau-3)b_n^2}{n}\bigg[\E\tr[(\T^{\frac 12}\D^{-1}\T^{\frac 12})\circ (\T^{\frac 12} \D^{-1}\bV^{-1}\T^{\frac 12})]\\
&+\frac{1}{p^2}\E\tr(\D^{-1}\T)\tr(\D^{-1}\bV^{-1}\T)\tr[\T\circ\T]\\
&-\frac{1}{p}\E\tr(\D^{-1}\T)\tr[(\T^{\frac 12}\D^{-1}\bV^{-1}\T^{\frac 12})\circ \T]
-\frac{1}{p}\E\tr(\D^{-1}\bV^{-1}\T)\tr[(\T^{\frac 12}\D^{-1}\T^{\frac 12})\circ \T]\bigg]\\
&+\frac{(\tau-3)b_n^3}{n^2}\bigg[\E\tr[(\T^{\frac 12}\D^{-1}\T^{\frac 12})\circ (\T^{\frac 12} \D^{-1}\T^{\frac 12})]
+\frac{1}{p^2}\E{\tr}^2(\D^{-1}\T)\tr[\T\circ \T]\\
&-\frac{2}{p}\E\tr(\D^{-1}\T)\tr[(\T^{\frac 12}\D^{-1}\T^{\frac 12})\circ \T]\bigg]\E \D^{-1}\bV^{-1}\T+o(1),\\
\uS_n&=\frac{-2b_n^3}{zn}\bigg[\E\tr \D^{-1}\T \D^{-1}\T+\frac{1}{p}\bigg( \frac 1p \tr\T^2\E\tr\T \D^{-1}\tr\T\D^{-1}
-2\E\tr\T^2\D^{-1}\tr\T\D^{-1}\bigg) \bigg]\\
&-\frac{(\tau-3)b_n^3}{zn}\bigg[\E\tr[(\T^{\frac 12}\D^{-1}\T^{\frac 12})\circ (\T^{\frac 12} \D^{-1}\T^{\frac 12})]
+\frac{1}{p^2}\E{\tr}^2(\D^{-1}\T)\tr[\T\circ \T]\\
&-\frac{2}{p}\E\tr(\D^{-1}\T)\tr[(\T^{\frac 12}\D^{-1}\T^{\frac 12})\circ \T]\bigg] +o(1). 
\end{align*}

To study the convergence of $S_n$ and $\uS_n$, 
we need to figure out the difference between $\D^{-1}$ and $\bV^{-1}$. Write 
\begin{eqnarray} 
\D^{-1}+\bV ^{-1} 
=b_n\R_1+\R_2+\R_3,\label{tilde-D-L} 
\end{eqnarray} 
where
\begin{align*} 
	\R_{1}&=\sum_{j=1}^n\bV^{-1}(\br_j\br_j'-n^{-1}\bSigma)\D_{j}^{-1},\quad
	\R_{2}=\sum_{j=1}^n\bV^{-1}\br_j\br_j'\D_{j}^{-1}(\beta_j-b_n),\\
	&\quad \quad \quad \quad \quad \quad \R_{3}=\frac{1}{n}\sum_{j=1}^nb_n\bV^{-1}\bSigma(\D_j^{-1}-\D^{-1}). 
\end{align*}
From \cite{BS04} we have, for any $p\times p$ matrix $\bM$, 
\begin{equation} 
|\E\tr \R_2\bM|\leq n^{1/2} K (\E||\bM||^4)^{1/4}\ \text{and}\ 
|\tr \R_3\bM|\leq K(\E||\bM||^2)^{1/2}\label{tilde-CM}
\end{equation} 
and, for nonrandom matrix $\bM$, 
\begin{align} 
|\E\tr \R_1\bM|\leq n^{1/2}K||\bM||\label{tilde-AM}. 
\end{align} 
Taking a step further, for $\bM$ nonrandom, we write  
\begin{equation} 
\tr \R_1\bSigma \D^{-1}\bM=\tilde R_{11}+\tilde R_{12}+\tilde R_{13}, 
\end{equation}
where 
\begin{align*} 
	\tilde R_{11}&=\tr \sum_{j=1}^n\bV ^{-1}\br_j\br_j'\D_j^{-1}\bSigma(\D^{-1}-\D_j^{-1})\bM,\\ 
	\tilde R_{12}&=\tr \sum_{j=1}^n\bV ^{-1}(\br_j\br_j'-n^{-1}\bSigma)\D_j^{-1}\bSigma\D_j^{-1}\bM,\\ 
	\tilde R_{13}&=-\frac{1}{n}\tr \sum_{j=1}^n\bV ^{-1}\bSigma\D_j^{-1}\bSigma(\D^{-1}-\D_j^{-1})\bM. 
\end{align*} 
It's clear that 
$\E\tilde R_{12}=0$ and moreover, 
using \eqref{D-bound}, \eqref{ineq-ex} and \eqref{tr-dm-bound}, we get 
\begin{align} 
|\E\tilde R_{13}|&\leq  K||\bM||,\\
\E\tilde R_{11}&=-n\E\beta_1\br_1\D_1^{-1}\bSigma \D_1^{-1}\br_1\br_1'\D_1^{-1}\bM\bV ^{-1}\br_1\nonumber\\ 
&=-b_nn^{-1}\E(\tr \D_1^{-1}\bSigma  \D_1^{-1}\T)( \tr \D_1^{-1}\bM\bV ^{-1}\bSigma )+o(1)\nonumber\\ 
&=-b_nn^{-1}\E(\tr \D^{-1}\bSigma  \D^{-1}\bSigma ) (\tr \D^{-1}\bM\bV^{-1}\bSigma )+o(1)\nonumber\\ 
&=-b_nn^{-1}\E(\tr \D^{-1}\bSigma  \D^{-1}\bSigma )\E (\tr \D^{-1}\bM\bV ^{-1}\bSigma )+o(1).\label{er11} 
\end{align}

Applying  \eqref{bnz}, \eqref{tilde-D-L}-\eqref{er11}, and Lemma \ref{st}, one may approximate each components of $S_n$ and $\uS_n$. Specifically, we  have 
\begin{align*} 
&\frac{1}{n}\E\tr \D^{-1}\T^k=-\int\frac{c_nt^kd \tilde H_p(t)}{z(1+\um_{0}t)}+o(1),\\
& \frac{1}{n}\E\tr \D^{-1}\bV^{-1}\T^k =-\int\frac{c_nt^kd \tilde H_p(t)}{z^2(1+\um_{0}t)^2}+o(1),\\
&\frac{1}{n}\E\tr \D^{-1}\T \D^{-1}\T \\
=&-\frac{1}{n}\E\tr \bV^{-1}\T \D^{-1}\T-\frac{b_n^2}{n^2} \E \tr \D^{-1}\T \D^{-1}\T \E\tr \bV^{-1}\T \D^{-1}\T+o(1)\nonumber\\ 
=&-\frac{1}{n}\E\tr \bV^{-1}\T \D^{-1}\T\left[1+\frac{b_n^2}{n} \E\tr \bV^{-1}\T \D^{-1}\T\right]^{-1}+o(1),\\
=&\int\frac{c_nt^2d \tilde H_p(t)}{z^2(1+\um_{0}t)^2}\left[1-\int\frac{c_n\um_0^2t^2d \tilde H_p(t)}{(1+\um_0t)^2}\right]^{-1}+o(1),\\
&\frac{1}{n}\E\tr \D^{-1}\T \D^{-1}\bV^{-1}\T \\
=&-\frac{1}{n}\E\tr \bV^{-1}\T \D^{-1}\bV^{-1}\T\left[1+\frac{b_n^2}{n}\E \tr \D^{-1}\T \D^{-1}\T\right]+o(1)\nonumber\\ 
=&-\frac{1}{n}\E\tr \bV^{-1}\T \D^{-1}\bV^{-1}\T\left[1+\frac{b_n^2}{n} \E\tr \bV^{-1}\T \D^{-1}\T\right]^{-1}+o(1),\\
=&\int\frac{c_nt^2d \tilde H_p(t)}{z^3(1+\um_0t)^3}\bigg[1-\int\frac{c_n\um_0^2t^2d \tilde H_p(t)}{(1+\um_0t)^2}\bigg]^{-1}+o(1).
\end{align*}

Combining the above results, we obtain
\begin{align*} 
&T_n/\uT_n=\int \frac{c_ntd \tilde H_p(t)}{z(1+\um_0 t)^2}+o(1),\\
&S_n-\uS_nT_n/\uT_n\\
=&-\int\frac{c_n\um_0^2t^2d \tilde H_p(t)}{z(1+\um_0t)^3}\bigg[1-\int\frac{c_n\um_0^2t^2d \tilde H_p(t)}{(1+\um_0t)^2}\bigg]^{-1}\\ 
&-\frac{2c_n\um_0^2}{z}\bigg[\int\frac{ (\alpha_2t-t^2)d \tilde H_p(t)}{1+\um_{0}t}\int\frac{td \tilde H_p(t)}{(1+\um_{0}t)^2} 
-\int\frac{td \tilde H_p(t)}{1+\um_{0}t}\int\frac{t^2d \tilde H_p(t)}{(1+\um_{0}t)^2}\bigg]\\
&-c_n(\tau-3)\bigg\{\frac{1}{z\um_0}g_{p,z}'\left(\frac{-1}{\um_0},\frac{-1}{\um_0}\right)
+\zeta_p \int\frac{t\um_0d \tilde H_p(t)}{1+\um_{0}t}\int\frac{t\um_0d \tilde H_p(t)}{z(1+\um_{0}t)^2}\\
&-\left[\int\frac{td \tilde H_p(t)}{z(1+\um_{0}t)}h_p'\left(\frac{-1}{\um_0}\right)
+\int\frac{t\um_0d \tilde H_p(t)}{z(1+\um_{0}t)^2}h_p\left(\frac{-1}{\um_0}\right)\right]
\bigg\}+o(1).
\end{align*}
Therefore, from \eqref{tnt} and
the identities 
\begin{equation*}
\bigg[1-\int \frac{c_ntd \tilde H_p(t)}{z(1+\um_0 t)^2}\bigg]^{-1}=-z\um_0\bigg[1-\int\frac{c_n\um_0^2 t^2d \tilde H_p(t)}{(1+\um_0 t)^2}\bigg]^{-1}=-\frac{z\um_0'}{\um_0}, 
\end{equation*} 
we obtain
\begin{align*} 
M_n^{(2)}(z)=\frac{S_n-\uS_nT_n/\uT_n}{1-T_n/\uT_n}=\mu_1(z)+(\tau-3)\mu_2(z)+o(1).
\end{align*}

The proof is complete.

\appendix \section{Additional lemmas}\label{appen}
In this Appendix, we present some lemmas and their proofs, which will be used in the proof of our main theorems.

\subsection{Lemmas}\label{less}

\begin{lemma}[Lemma 2.7 in \cite{BS98}]\label{bs98}
 For \(\z=(z_{1}, \ldots, z_p)'\) i.i.d.\ standardized entries, \(\C\) \(p\times p\) matrix (complex) we have for any \(k \geq 2\)
$$
\E\left|\z' \C \z-\operatorname{tr} \C\right|^{k} \leq K \left[\left(\E|z_1|^{4} \operatorname{tr} \C\C^{*}\right)^{\frac k2}+\E|z_1|^{2 k} \operatorname{tr}\left(\C \C^{*}\right)^{\frac k2}\right],
$$
where $K$ is a constant depending only on $k$. 
\end{lemma}

\begin{lemma}\label{double-e} 
 Suppose that Assumptions (a)-(c) and \eqref{tcr} hold. Let $\z=(z_1,\ldots,z_p)'$ be a random vector with i.i.d.\ standardized entries,
  $\br=\sqrt{p/n}\A^{1/2}\z/||\A^{1/2}\z||$, and $\bSigma=n\E\br\br'$.
	Then for any $p\times p$ complex matrices $\C$ and $\tilde \C$ with bounded spectral norms, 
	\begin{align*} 
	&n^2\E\left(\br'\C\br-\frac1n\tr\bSigma \C\right)\left(\br'\tilde \C\br-\frac1n\tr\bSigma\tilde \C\right)\\ 
=&\tr\T  \C\T  \tilde \C +\tr\T  \C\T  \tilde \C' 
+\frac{2}{p^2}\tr\T ^2\tr\T  \C\tr\T  \tilde \C
-\frac{2}{p}\tr\T ^2\C\tr\T \tilde \C
-\frac{2}{p}\tr\T  \C\tr \T ^2\tilde \C\\
&+(\tau-3)\bigg\{ \tr[(\T ^{\frac 12}\C\T ^{\frac 12})\circ (\T ^{\frac 12}\tilde \C\T ^{\frac 12})]+\frac{1}{p^2}\tr\C\T \tr\tilde\C\T \tr[\T \circ\T ]\\
&-\frac{1}{p}\tr\C\T \tr[(\T ^{\frac 12}\tilde\C\T ^{\frac 12})\circ \T ]-\frac{1}{p}\tr\tilde \C\T \tr[(\T ^{\frac 12}\C\T ^{\frac 12})\circ \T ]\bigg\}+o(p). 
	\end{align*}  
\end{lemma}

\begin{lemma}\label{st}
 Suppose that Assumptions (a)-(c) and \eqref{tcr} hold with $\gamma=5$. We have
\begin{align*}
||\bSigma-\T||=o(p^{-1}),
\end{align*}
where $\bSigma$ is defined in Lemma \ref{double-e} and $\T$ is given in \eqref{tre}.
\end{lemma}

\begin{lemma}\label{ineq} 
Under the assumptions in Lemma \ref{double-e}, for any $k\geq 2$, 
	\begin{align}\label{q-moment} 
	\E\left|\br'\C\br-\frac1n\tr \bSigma \C\right|^k
	&\leq Kn^{-k}\left[\E|z_1|^{2 k} \operatorname{tr} (\C\bSigma)^{k}+\left(\E|z_1|^{4} \operatorname{tr} (\C\bSigma)^{2}\right)^{\frac k2}+\|\C\bSigma\|^{k}\left(p^{\frac k2} \E^{\frac k2}|z_1|^{4}+p \E|z_1|^{2 k}\right)\right]\nonumber\\
	&\leq K\left(n^{-\frac k2} +n^{-1- \frac{k(\gamma-4)}{\gamma}}\right).
	\end{align} 
	where $K$ is a constant depending only on $k$. 
\end{lemma}

\subsection{Proof of Lemma {\ref{double-e}}}

Denote $\W=\A^{\frac 12}\C \A^{\frac 12}$, $\U=\A^{\frac 12}\tilde \C\A^{\frac 12}$, and $s=\z' \A\z/p$. We consider the product of the quadratic form $n^2\br'\C\br\br'\tilde \C\br=\z'\W\z\z'\U\z/s^2$. 
From Lemma \ref{bs98}, the fact $\tr \A=p$, and the conditions in \eqref{tcr}, it holds that  
\begin{align}\label{sk}
\E|s-1|^k\leq K\left(p^{-\frac k2}+p^{-1-\frac{k(\gamma-4)}{\gamma}}\right),\quad k\geq 2.
\end{align}
By the identity 
$$\frac{1}{s^2}=2-s^2+(1-s^2)^2+s^{-2}(1-s^2)^3$$
and the inequality
$$
\E (\z'\W\z\z'\U\z)(s^{-2}(1-s^2)^3)\leq Kp^2\E |1-s|^3=o(p),
$$
we have
\begin{align}\label{yy}
\E \br'\C\br\br'\tilde \C\br
&=\E (\z'\W\z\z'\U\z)\Big(6-8s+3s^2\Big)+o(p).
\end{align}
Therefore, the main task in the following is to derive the limits for the three terms $\E \z'\W\z\z'\U\z$, $\E \z'\W\z\z'\U\z s$ and $\E \z'\W\z\z'\U\z s^2$ up to the order $O(p)$.

For the first term  $\E \z'\W\z\z'\U\z$, we have 
\begin{align*}
\E \z'\W\z\z'\U\z=\E \sum_{i, j, k, \ell} z_i z_j z_k z_\ell\W_{ij}\U_{k\ell}.
\end{align*}
Since all the $p$ components $z_i$ are independent and  standardized, with mean zero, variance one and finite fourth moment,  the terms that will contribute are the ones with their indexes  either can be glued together   or divided into two groups, i.e. $i=j=k=\ell$, or $i=j\neq k=\ell$, or $i=k\neq j=\ell$ or $i=\ell \neq j=k$. All the four cases together gives 
\begin{align}\label{zwzu}
\E \z'\W\z\z'\U\z=\tr \W\tr\U+\tr \W\U+\tr \W'\U+(\tau-3)\sum_i \W_{ii}\U_{ii}+o(p).
\end{align}
For the second term  $\E \z'\W\z\z'\U\z s$, we have 
\begin{align}\label{rfd}
\E \z'\W\z\z'\U\z s=
\frac 1p\E \sum_{i, j, k, \ell, s, u}  z_i z_j z_k z_\ell z_s z_u\W_{ij}\U_{k\ell} \A_{su}.
\end{align}
The terms that will contribute up to order $O(p)$ are  in $\sum_{(2)}$ and $\sum_{(3)}$, where the index $(\cdot)$ denotes the number of distinct integers in the set $\{i, j, k, \ell, s, u\}$. It can be checked that the following three cases should be counted in $\sum_{(2)}$ (all have the form of the product of two traces)

\noindent  case 1: $i=j\neq k=\ell=s=u$,

\noindent  case 2: $k=\ell\neq i=j=s=u$,

\noindent  case 3: $s=u\neq i=j=k=\ell$, 

\noindent while in $\sum_{(3)}$ the following four cases should be taken into account,

\noindent  case 1: $k=s\neq \ell=u\neq i=j$ and $k=u\neq \ell=s\neq i=j$,

\noindent  case 2: $i=s\neq j=u\neq k=\ell$ and $i=u\neq j=s\neq k=\ell$,

\noindent  case 3: $i=\ell\neq j=k\neq s=u$ and $i=k\neq j=\ell\neq s=u$,

\noindent  case 4: $i=j\neq k=\ell\neq s=u$.

\noindent Combining the contribution of each cases in $\sum_{(2)}$ and $\sum_{(3)}$, we have 
\begin{align*}
\text{case 1}&=\frac{\tau+o(1)}{p}\sum_{i\neq k}\W_{ii}\U_{kk}\A_{kk}+\frac{2}{p}\sum_{i\neq k\neq \ell}\W_{ii}\U_{k\ell}\A_{\ell k}\\
&=\frac{\tau-2}{p}\sum_{i\neq k}\W_{ii}\U_{kk}\A_{kk}+\frac{2}{p}\sum_{i\neq k}\W_{ii}(\U\A)_{k k}+o(p)\\
&=\frac {\tau-2}{p} \tr \W\sum_k \U_{kk}\A_{kk}+\frac 2p \tr \W \tr (\U\A)+o(p),\\
\text{case 2}&=\frac{\tau+o(1)}{p}\sum_{i\neq k}\W_{ii}\U_{kk}\A_{ii}+\frac{2}{p}\sum_{i\neq j\neq k}\W_{ij}\U_{kk}\A_{ji}\\
&=\frac{\tau-2}{p}\sum_{i\neq k}\W_{ii}\A_{ii}\U_{kk}+\frac{2}{p}\sum_{i\neq k}\U_{k k}(\W\A)_{ii}+o(p)\\
&=\frac {\tau-2}{p} \tr \U\sum_i \W_{ii}\A_{ii}+\frac 2p \tr \U \tr (\W\A)+o(p),\\
\text{case 3}&=\frac{\tau+o(1)}{p}\sum_{s\neq i}\W_{ii}\U_{ii}\A_{ss}+\frac{1}{p}\sum_{i\neq j\neq s}\W_{ij}\U_{ji}\A_{ss}+\frac{1}{p}\sum_{i\neq j\neq s}\W_{ij}\U^*_{ji}\A_{ss}\\
&=\frac{\tau-2}{p}\sum_{s\neq i}\W_{ii}\U_{ii}\A_{ss}+\frac{1}{p}\sum_{i\neq s}\A_{ss}(\W\U)_{ii}+\frac{1}{p}\sum_{i\neq s}\A_{ss}(\W\U^*)_{ii}+o(p)\\
&=\frac {\tau-2}{p} \tr \A\sum_i \W_{ii}\U_{ii}+\frac 1p \tr \A \tr (\W\U)+\frac 1p \tr \A \tr (\W\U^*)+o(p),\\
\text{case 4}&=\frac 1p \sum_{i\neq k \neq s}\W_{ii}\U_{kk}\A_{ss}\\
&=\frac 1p \tr \W \tr \U \tr \A-\frac 1p \tr \A\sum_i \W_{ii}\U_{ii}-\frac 1p \tr \U\sum_i \W_{ii}\A_{ii}\\
&\quad-\frac 1p \tr \W\sum_i \A_{ii}\U_{ii}+o(p),
\end{align*}
which further gives
\begin{align}\label{zzzzs}
\E \z'\W\z\z'\U\z s&=\text{case 1}+\text{case 2}+\text{case 3}+\text{case 4}+o(p)\nonumber\\
&=\frac 1p \tr \W \tr \U \tr \A+\frac 2p \tr \W \tr (\U\A)+\frac 2p \tr \U \tr (\W\A)\nonumber\\
&\quad+\frac 1p \tr \A \tr (\W\U)+\frac 1p \tr \A \tr (\W\U^*)+\frac {\tau-3}{p} \tr \W\sum_k \U_{kk}\A_{kk}\nonumber\\
&\quad+\frac {\tau-3}{p} \tr \U\sum_i \W_{ii}\A_{ii}+\frac {\tau-3}{p} \tr \A\sum_i \W_{ii}\U_{ii}+o(p).
\end{align}
Finally, for the third term  $\E \z'\W\z\z'\U\z s^2$, we have 
\begin{align*}
\E \z'\W\z\z'\U\z s^2=
\frac {1}{p^2}\E \sum_{i, j, k, \ell, s, u, m, b}  z_i z_j z_k z_\ell z_s z_uz_m z_b\W_{ij}\U_{k\ell} \A_{su}\A_{mb}.
\end{align*}
The terms that will make the main contribution up to order $O(p)$ are  in $\sum_{(3)}$ and $\sum_{(4)}$. For example, when considering $\sum_{(1)}$, we have 
\begin{align*}
\sum_{(1)}=\E \sum_{i} \frac {1}{p^2} z^8_i \W_{ii}\U_{ii} \A^2_{ii}=O(p^{1-4(\gamma-4)/\gamma})=o(p)
\end{align*}
by using the assumptions in \eqref{tcr}.
Similar technique can be applied for dealing with the terms in  $\sum_{(2)}$ and get the $o(p)$ bound, thus can be neglected. For terms in $\sum_{(3)}$ and $\sum_{(4)}$, we list in the following all the cases that should be counted, which are all up to order $O(p)$. 
For $\sum_{(3)}$, we have six cases

\noindent  case 1: $i=j\neq k=\ell\neq s=u=m=b$,

\noindent  case 2: $i=j=s=u\neq k=\ell\neq m=b$,

\noindent  case 3: $i=j=m=b\neq k=\ell\neq s=u$,

\noindent  case 4: $k=\ell=m=b\neq i=j\neq s=u$,

\noindent  case 5: $i=j=k=\ell\neq s=u\neq m=b$,

\noindent  case 6: $k=\ell=s=u\neq i=j\neq m=b$,

\noindent while in $\sum_{(4)}$, we have seven cases 

\noindent case 1: $i=j\neq k=\ell\neq u=m\neq s=b$ and $i=j\neq k=\ell\neq s=m\neq u=b$,

\noindent  case 2: $i=s\neq j=u\neq k=\ell\neq m=b$ and $i=u\neq j=s\neq k=\ell\neq m=b$,

\noindent  case 3: $i=m\neq j=b\neq u=s\neq k=\ell$ ang $i=b\neq j=m\neq k=\ell\neq s=u$,

\noindent  case 4: $k=m\neq \ell=b\neq i=j\neq s=u$ and $k=b\neq \ell=m\neq i=j\neq s=u$,

\noindent  case 5: $i=k\neq j=\ell\neq s=u\neq m=b$ and $i=\ell\neq j=k\neq s=u\neq m=b$,

\noindent  case 6: $k=s\neq \ell=u\neq i=j\neq m=b$ and $k=u\neq \ell=s\neq i=j\neq m=b$,

\noindent case 7:  $i=j\neq k=\ell\neq s=u\neq m=b$. 

\noindent Combining the above, we have 
\begin{align*}
\text{case 1}&=\frac{2}{p^2}\sum_{i\neq k\neq m\neq s}\W_{ii}\U_{kk}\A_{ms}\A_{ms}+\frac{\tau+o(1)}{p^2}\sum_{i\neq k\neq s}\W_{ii}\U_{kk}\A^2_{ss}\\
&=\frac{2}{p^2}\sum_{i\neq k\neq s}\W_{ii}\U_{kk}(\A\A)_{ss}+\frac{\tau-2}{p^2}\sum_{i\neq k\neq s}\W_{ii}\U_{kk}\A^2_{ss}+o(p)\\
&=\frac{2}{p^2}\tr \A^2 \tr \W \tr \U+\frac{\tau-2}{p^2} \tr \W \tr\U \sum_s \A^2_{ss}+o(p),\\
\text{case 2}&=\frac{2}{p^2}\sum_{i\neq j\neq k\neq m}\W_{ij}\U_{kk}\A_{ij}\A_{mm}+\frac{\tau+o(1)}{p^2}\sum_{i\neq k\neq m}\W_{ii}\U_{kk}\A_{ii}\A_{mm}\\
&=\frac{2}{p^2}\sum_{i\neq k\neq m}(\W\A)_{ii}\U_{kk}\A_{mm}+\frac{\tau-2}{p^2}\sum_{i\neq k\neq m}\W_{ii}\U_{kk}\A_{ii}\A_{mm}+o(p)\\
&=\frac{2}{p^2}\tr (\W\A) \tr \U \tr \A+\frac{\tau-2}{p^2} \tr \U \tr\A \sum_i \W_{ii}\A_{ii}+o(p),\\
\text{case 3}&=\text{case 2},\\
\text{case 4}&=\frac{2}{p^2}\sum_{k\neq \ell \neq i\neq s}\W_{ii}\U_{k\ell}\A_{ss}\A_{k\ell}+\frac{\tau}{p^2}\sum_{k\neq i\neq s}\W_{ii}\U_{kk}\A_{kk}\A_{ss}\\
&=\frac{2}{p^2}\sum_{k\neq i\neq s}(\U\A)_{kk}\W_{ii}\A_{ss}+\frac{\tau-2}{p^2}\sum_{k\neq i\neq s}\W_{ii}\U_{kk}\A_{ss}\A_{kk}+o(p)\\
&=\frac{2}{p^2}\tr (\U\A) \tr \W \tr \A+\frac{\tau-2}{p^2} \tr \W \tr\A \sum_k \U_{kk}\A_{kk}+o(p),\\
\text{case 5}&=\frac{1}{p^2}\sum_{i\neq j \neq s \neq m}\W_{ij}\U_{ij}\A_{ss}\A_{mm}+\frac{1}{p^2}\sum_{i\neq j \neq s \neq m}\W_{ij}\U_{ji}\A_{ss}\A_{mm}\\
&\quad+\frac{\tau}{p^2}\sum_{i\neq s\neq m}\W_{ii}\U_{ii}\A_{ss}\A_{mm}\\
&=\frac{1}{p^2}\sum_{i\neq s\neq m}(\W\U)_{ii}\A_{ss}\A_{mm}+\frac{1}{p^2}\sum_{i\neq s\neq m}(\W\U^*)_{ii}\A_{ss}\A_{mm}\\
&\quad+\frac{\tau-2}{p^2}\sum_{i\neq s\neq m}\W_{ii}\U_{ii}\A_{ss}\A_{mm}+o(p)\\
&=\frac{1}{p^2}\tr (\W\U) (\tr \A)^2+\frac{1}{p^2}\tr (\W\U^*) (\tr \A)^2+\frac{\tau-2}{p^2} (\tr \A)^2\sum_i \W_{ii}\U_{ii}+o(p),\\
\text{case 6}&=\text{case 4},\\
\text{case 7}&=\frac{1}{p^2}\sum_{i\neq k\neq s\neq m}\W_{ii}\U_{kk}\A_{ss}\A_{mm}\\
&=\frac{1}{p^2}\tr\W \tr \U(\tr \A)^2-\frac{1}{p^2}\tr \W\tr \U\sum_s \A^2_{ss}-\frac{2}{p^2}\tr \W \tr \A \sum_s \A_{ss}\U_{ss}\\
&\quad -\frac{1}{p^2}(\tr \A)^2 \sum_i \W_{ii}\U_{ii}-\frac{2}{p^2}\tr \U \tr \A \sum_i \W_{ii}\A_{ii}+o(p),
\end{align*}
which finally leads to 
\begin{align}\label{zzzzss}
&~\quad\E (\z'\W\z\z'\U\z)s^2\nonumber\\
&=\frac{1}{p^2}\tr\W \tr \U(\tr \A)^2+\frac{2}{p^2}\tr \A^2 \tr \W \tr \U+\frac{4}{p^2}\tr (\W\A) \tr \U \tr \A\nonumber\\
&\quad+\frac{4}{p^2}\tr (\U\A) \tr \W \tr \A+\frac{1}{p^2}\tr (\W\U) (\tr \A)^2+\frac{1}{p^2}\tr (\W\U^*) (\tr \A)^2\nonumber\\
&\quad+\frac{\tau-3}{p^2} \tr \W \tr\U \sum_s \A^2_{ss}+\frac{2\tau-6}{p^2} \tr \U \tr\A \sum_i \W_{ii}\A_{ii}\nonumber\\
&\quad+\frac{2\tau-6}{p^2} \tr \W \tr\A \sum_k \U_{kk}\A_{kk}+\frac{\tau-3}{p^2} (\tr \A)^2\sum_i \W_{ii}\U_{ii}+o(p).
\end{align}

Collecting \eqref{yy}, \eqref{zwzu}, \eqref{zzzzs}, and \eqref{zzzzss}, we have 
\begin{align}\label{twoqu}
&\quad~n^2\E \br'\C\br\br'\tilde \C\br\nonumber\\
&=(\tau-3)\sum_i \W_{ii}\U_{ii}+\tr \W \tr U+\tr (\W\U)+\tr (\W'\U)\nonumber\\
&\quad+\frac{6}{p^2}\tr \A^2\tr \W\tr \U-\frac{4}{p}\tr (\W\A)\tr \U-\frac{4}{p}\tr (\U\A)\tr \W\nonumber\\
&\quad+\frac{3(\tau-3)}{p^2}\tr \W \tr \U\sum_s\A^2_{ss}-\frac{2(\tau-3)}{p}\tr \W \sum_k\U_{kk}\A_{kk}\nonumber\\
&\quad-\frac{2(\tau-3)}{p}\tr \U \sum_i\W_{ii}\A_{ii}+o(p).
\end{align}
On the other hand,  using the identity
\begin{align*}
\frac 1s=2-s+(1-s)^2+s^{-1}(1-s)^3
\end{align*}
and the inequality \eqref{sk}, we can derive 
\begin{align}\label{ycy}
n\E \br'\C\br=\E\frac 1s \z'\W\z=\E \z'\W\z\Big( 3-3s+s^2\Big)+o(1).
\end{align}
It is trivial to have 
\begin{align}\label{zz}
\E \z'\W\z=\tr \W
\end{align}
and by applying \eqref{zwzu} and  \eqref{zzzzs} again, 
\begin{align}
\E \z' \W \z s&=\frac{\tau-3}{p}\sum_i \W_{ii}\A_{ii}+\tr \W +\frac 1p \tr(\W\A)+\frac 1p \tr(\W^*\A),\label{zss1}\\
\E \z' \W \z s^2&=\tr \W+\frac{2}{p^2}\tr \W \tr(\A^2)+\frac{4}{p}\tr (\W\A)+\frac{2(\tau-3)}{p}\sum_i\W_{ii}\A_{ii}\label{zss2}\\
&\quad+\frac{\tau-3}{p^2}\tr\W\sum_i \A^2_{ii}+o(1).\nonumber
\end{align}
Collecting \eqref{ycy}-\eqref{zss2} leads to
\begin{align}\label{ex}
n\E \br'\C\br&= \tr \W+\frac{\tau-3}{p^2}\tr \W\sum_i \A^2_{ii}+\frac{2}{p^2}\tr\W\tr \A^2-\frac{\tau-3}{p}\sum_i\W_{ii}\A_{ii}\nonumber\\
&\quad-\frac{2}{p}\tr(\W\A)+o(1).
\end{align}
Therefore, combining \eqref{twoqu}-\eqref{ex},
we have reached 
\begin{align} \label{cc1}
&n^2\E\left(\br'\C\br-\frac 1n\tr\bSigma \C\right)\left(\br'\tilde \C\br-\frac1n\tr\bSigma\tilde \C\right)\\
=&n^2\E\br'\C\br\br'\tilde \C\br-n^2\E\br'\C\br\E\br'\tilde \C\br\nonumber\\ 
=&\tr [(\W'+ \W)\U]+\frac{2}{p^2}\tr \A^2\tr\W\tr \U
-\frac{2}{p}\tr (\W\A)\tr \U\nonumber\\
&-\frac{2}{p}\tr (\U\A)\tr\W
+(\tau-3)\tr(\W\circ\U)+\frac{\tau-3}{p^2}\tr\W\tr \U \tr(\A\circ\A)\nonumber\\
&-\frac{\tau-3}{p}\tr\W \tr(\U\circ\A)-\frac{\tau-3}{p}\tr \U \tr(\W\circ\A)+o(p).\nonumber
\end{align}
The proof is then complete by replacing all the matrix $\A$ with $\T$.

\subsection{Proof of Lemma \ref{st}}

Using  the identity
\begin{align*}
\frac 1s=2-s+(1-s)^2+s^{-1}(1-s)^3
\end{align*}
 we have
\begin{align}\label{ew2}
\bSigma=\E \frac 1s \A^{\frac 12}\z \z'\A^{\frac 12}=\E  \A^{\frac 12}\z \z'\A^{\frac 12}\Big(2-s+(1-s)^2+s^{-1}(1-s)^3\Big),
\end{align}
where $s=\z' \A\z/p$. 
First we show that 
\begin{align}\label{eqw1} 
\left\|\E  \A^{\frac 12}\z \z'\A^{\frac 12}s^{-1}(1-s)^3\right\|=o(p^{-1}).
\end{align}
Define an event $A=\big\{|s-1|>1/2\big\}$ then, by Markov's inequality and \eqref{sk}, we have 
${\rm P}(A)=o(n^{-s})$ for any $s>0$.
Therefore,
\begin{align*}
\Big\|\E  \A^{\frac 12}\z \z'\A^{\frac 12}s^{-1}(1-s)^3\Big\|&\leq K \Big\|\E  \z \z's^{-1}(1-s)^3I(A)\Big\|+K \Big\|\E  \z \z's^{-1}(1-s)^3I(A^c)\Big\|\\
&\leq K \Big\|\E  \z \z' |1-s|^3\Big\|+o(n^{-s}).
\end{align*}
Applying H\"older's inequality and \eqref{sk}, we have
\begin{align*}
\Big\|\E  \z \z' |1-s|^3\Big\|&=\max_{\alpha\in \mathbb R^p,\|\alpha\|=1}\E \alpha' \z \z' \alpha |1-s|^3
\leq \max_{\alpha\in \mathbb R^p,\|\alpha\|=1}\E \big|\z' \alpha \alpha' \z -1\big|\big|1-s\big|^3+\E\big|1-s\big|^3\\
&\leq \max_{\alpha\in \mathbb R^p,\|\alpha\|=1}\E^{\frac 12} \big|\z' \alpha \alpha' \z -1\big|^2\E^{\frac 12}\big|1-s\big|^6+o(p^{-1}),
\end{align*}
which is $o(p^{-1})$ from \eqref{sk} and the fact  $\E |\z' \alpha \alpha' \z -1|^2=O(1)$. Therefore, \eqref{eqw1} is verified, which together with \eqref{ew2} give
\begin{align}\label{com1}
\bSigma=\E  \A^{\frac 12}\z \z'\A^{\frac 12}\Big(2-s+(1-s)^2\Big)+o(p^{-1})=\A^{\frac 12}\left[\E  \z \z'\Big(3-3s+s^2\Big)\right]\A^{\frac 12}+o(p^{-1}),
\end{align}
where (and in the following) the ``$o(p^{-1})$" is in terms of spectral norm.

Next, we deal with  the terms  $\E \z \z's$ and $\E \z \z' s^2$.
For $\E \z \z' s$, we have its $(i,j)$-th entry given by
\begin{align*}
[\E  \z \z' s]_{ (i, j)}&=\frac 1p \E z_i z_j \sum_{k, \ell} z_k z_\ell \A_{k \ell}
=\left\{\begin{array}{ll}
\frac 1p \A_{ij}+\frac 1p \A_{ji} & i\neq j\\
1+\frac 1p\Big(\tau-1+o(1)\Big)\A_{ii} & i=j
\end{array}\right.,
\end{align*}
which gives  
\begin{align*}
\E  \z \z' s=\I_p+\frac 2p \A+\frac 1p \Big(\tau-3+o(1)\Big)\text{diag}(\A)
\end{align*}
and
\begin{align}\label{com2}
\E  \A^{\frac 12}\z \z'\A^{\frac 12} s=\A+\frac 2p \A^2+\frac 1p \Big(\tau-3\Big)\A^{\frac 12} \text{diag}(\A)\A^{\frac 12}+o(p^{-1}).
\end{align}
For the term $\E  \z \z' s^2$, similar to the derivation of \eqref{rfd}, we have its $(i,j)$-th entry is given by 
\begin{align*}
[\E  \z \z' s^2]_{ (i, j)}&=\frac{1}{p^2} \E z_i z_j \sum_{k, \ell, s, u} z_k z_\ell z_s z_u \A_{k \ell} \A_{su}\\
&=\left\{\begin{array}{ll}
\frac 4p \A_{ij}-\frac{4}{p^2}\A_{ii}\A_{ij}-\frac{4}{p^2}\A_{ij}\A_{jj}+o(p^{-2}) & i\neq j\\
\frac{1}{p^2}\big(\tau-3\big)\sum_k \A^2_{kk}+\frac{2}{p}(\tau-1)\A_{ii}+\frac{2}{p^2}\tr \A^2+1+o(p^{-1}) &i=j
\end{array}\right..
\end{align*} 
Therefore, we get
\begin{align*}
\E  \z \z' s^2=\frac 4 p\A+\frac{1}{p^2}(\tau-3)\tr(\A\circ \A) \I_p+\I_p+\frac{2}{p}(\tau-3)\text{diag}(\A)+\frac{2}{p^2}\tr \A^2 \cdot \I_p+o(p^{-1}), 
\end{align*}
which further gives that 
\begin{align}\label{com3}
\E  \A^{\frac 12}\z \z' \A^{\frac 12} s^2=\frac 4 p\A^2+\frac{1}{p^2}(\tau-3)\tr(\A\circ \A) \A+\A+\frac{2}{p}(\tau-3)\A^{\frac 12}\text{diag}(\A)\A^{\frac 12}+\frac{2}{p^2}\tr \A^2 \cdot \A+o(p^{-1}).
\end{align}
Collecting \eqref{com1}, \eqref{com2}, and \eqref{com3}, we obtain 
\begin{align*}
\bSigma=\A
-\frac{\tau-3}{p}\A{\rm diag} (\A)\A'-\frac{2}{p}\A^2
+\left(\frac{\tau-3}{p^2}\tr(\A\circ\A)+\frac{2}{p^2}\tr \A^2\right)\A+o(p^{-1}).
\end{align*}
The proof is thus complete.

\subsection{Proof of Lemma \ref{ineq}}

This lemma can be obtained from similar arguments for the proof of Lemma 6 in \cite{MJ19}. We omit the details.

\section*{Acknowledgement}

Weiming Li's research is partially supported by NSFC (No.\ 11971293 ) and Program of IRTSHUFE. Qinwen Wang acknowledges support from a NSFC    Grant (No.\ 11801085) and the Shanghai Sailing Program (No.\  18YF1401500).
Jianfeng Yao's research is partly supported by a HKSAR RGC Grant (GRF
17306918).  
Wang Zhou's research is partially supported by the MOE Tier 2 grant MOE2015-T2-2-039 (R-155-000-171-112).


\end{document}